\documentclass[12pt,a4paper]{article}
\usepackage{verbatim, latexsym, amssymb, amsmath, amsthm, amsbsy,color, mathabx, amscd}
\usepackage[shortlabels]{enumitem}
\usepackage{epsfig}
\usepackage{tikz-cd}

\usepackage{geometry}

\usepackage{hyperref}
\newtheorem*{thm*}{Theorem}

\newtheorem{thm}{Theorem}[section]
\newtheorem{prop}[thm]{Proposition}
\newtheorem{lem}[thm]{Lemma}
\newtheorem{cor}[thm]{Corollary}

\newtheorem{que}[thm]{Question}
\theoremstyle{definition}
\newtheorem{defn}[thm]{Definition}
\newtheorem{rem}[thm]{Remark}

\newtheorem{construction}[thm]{Construction}

\newcommand{\Diff}{\mathrm{Diff}}
\newcommand{\Ham}{\mathrm{Ham}}

\def\Homeo{\mathrm{Homeo}}
\def\Hameo{\mathrm{Hameo}}
\def\FHomeo{\mathrm{FHomeo}}

\def\id{\mathrm{Id}}

\def\Cal{\mathrm{Cal}}
\def\Ru{\mathrm{Ru}}

\def\supp{\mathrm{supp}}
\def\val{\mathrm{val}}

\newcommand{\Sym}{\operatorname{Sym}}

\newcommand{\bD}{\mathbb{D}}
\newcommand{\bR}{\mathbb{R}}
\newcommand{\R}{\mathbb{R}}
\newcommand{\D}{\mathbb{D}}
\renewcommand{\S}{\mathbb{S}}

\newcommand{\bZ}{\mathbb{Z}}

\newcommand{\ul}{\underline}

\definecolor{sobhan}{rgb}{0,.6,0}
\definecolor{vincent}{rgb}{1.0,0.3,0}
\definecolor{dan}{rgb}{.5,.2,1}
\definecolor{cym}{rgb}{.7,0,.4}

\begin{document}
\title{Subleading asymptotics of link spectral invariants and homeomorphism groups of surfaces}

\author{Dan Cristofaro-Gardiner, Vincent Humili\`ere, Cheuk Yu Mak
\\Sobhan Seyfaddini and Ivan Smith}

\maketitle

\begin{abstract}
This paper continues the study of  link spectral invariants on compact surfaces, introduced in our previous work and shown to satisfy a Weyl law in which they asymptotically recover the Calabi invariant. Here we study their subleading asymptotics on surfaces of genus zero. We show the subleading asymptotics are bounded for smooth time-dependent Hamiltonians, and recover the Ruelle invariant for autonomous disc maps with finitely many critical values. We deduce that the Calabi homomorphism admits infinitely many extensions to the group of compactly supported area-preserving homeomorphisms, and that the kernel of the Calabi homomorphism on the group of hameomorphisms is not simple.

%
%
%
\end{abstract}

\textsc{French title:} ``Asymptotique sous-dominante des invariants spectraux d'entrelacs et groupes d'homéomorphismes de surfaces.''

\textsc{French abstract}

Cet article poursuit l'étude des invariants spectraux d'entrelacs introduits dans notre précédent travail, dans lequel il est établi qu'ils vérifient une loi de Weyl faisant apparaître l'invariant de Calabi asymptotiquement. Nous étudions ici leur asymptotique sous-dominante sur les surfaces de genre nul. Nous montrons que celle-ci est bornée pour tous les hamiltoniens lisses dépendant du temps, et qu'elle fait apparaître l'invariant de Ruelle pour les hamiltoniens autonomes du disque ayant un nombre fini de valeurs critiques. Nous en déduisons que le morphisme de Calabi admet une infinité d'extensions au groupe des homéomorphismes à support compact qui préservent l'aire, et que le noyau du morphisme de Calabi sur le groupe des haméomorphismes n'est pas simple.

\setcounter{tocdepth}{1}
\tableofcontents


\section{Introduction}

\subsection{Area-preserving homeomorphisms of surfaces}
Let $(M,\omega)$ be a compact manifold possibly with boundary, equipped with a volume-form, and consider the group $\Homeo_{\mathrm{c}}(M,\omega)$ of volume-preserving homeomorphisms that are the identity near the boundary, in the component of the identity.  

When the dimension of $M$ is at least three, there is a clear picture due to Fathi regarding the algebraic structure of this group: there is a mass-flow homomorphism, and its kernel is a simple group.  In contrast, in dimension two the situation is much less understood despite the fact that many decades have passed since Fathi's work. 

We refer the reader to \cite{CGHMSS21} for background and a definition of the mass-flow homomorphism (op.cit. Sec 2.3).
In the cases that we will be mainly concerned with here, that is the disc and the sphere, the kernel of mass-flow is just the group of area and orientation preserving homeomorphisms.

We recently showed \cite[Thm. 1.3]{CGHMSS21} that when $\dim(M) = 2$, the kernel of mass-flow is never simple.   In fact, it contains as a proper normal subgroup the group $\Hameo(M,\omega)$ of {\em hameomorphisms}, whose definition we review in Definition \ref{def:hameo}.  When $M$ has boundary, we also showed that the classical Calabi homomorphism, which we review in Definition \ref{def:Calabi} and which measures the average rotation of the map, extends to $\Hameo$ from the group $\Ham_{\mathrm{c}}(M,\omega)$ of Hamiltonian diffeomorphisms that are the identity near the boundary.  It is then natural to ask the following.

\begin{que}
\label{que:main}
When $M$ is closed, is $\Hameo(M,\omega)$ simple?  When $M$ has non-empty boundary, is the kernel of Calabi on $\Hameo(M,\omega)$ simple?
\end{que}

This is an old question.  For example, a variant appears in 
\cite[Problem (4)]{muller-oh}.  Let us briefly explain why one might hope for a positive answer.  Hameomorphisms are homeomorphisms with well-defined Hamiltonians, and it is natural to wonder whether the algebraic structure of the group of hameomorphisms could be like that of the group $\Ham_{\mathrm{c}}(M,\omega)$;  moreover, Banyaga showed \cite{Banyaga} that $\Ham_{\mathrm{c}}$ is simple when $M$ is closed and the kernel of Calabi is simple when $M$ has boundary.  

Our first result shows that the structure of $\Hameo$ is more complicated than this.

\begin{thm}
\label{thm:notsimple} The following groups are not perfect:
\begin{enumerate}
\item The kernel of Calabi on $\mathrm{Hameo}(\mathbb{D}^2,\omega)$.
\item The group $\mathrm{Hameo}(\mathbb{S}^2,\omega)$.
\end{enumerate}
Both admit surjective group homomorphisms to $\mathbb{R}$.
\end{thm}

Recall that a group $G$ is called {\em perfect} if it coincides with its commutator subgroup $[G,G]$.  Note that, since the commutator subgroup is always normal, every (non-abelian) simple group is perfect and hence we conclude that neither of the groups appearing in the above theorem are simple.

\subsection{A two-term Weyl law}

Theorem~\ref{thm:notsimple} is proved by studying the asymptotics of the ``link spectral invariants'' defined in our previous work  \cite[Thm. 1.13, Def. 6.14, Eq. (59)]{CGHMSS21}. In Sec. 7.3 of the same work 
we defined quasimorphisms
\[ \mu_k : \Diff(\S^2,\omega) \to \mathbb{R}, \quad f_k: \Homeo_{\mathrm{c}}(\D^2,\omega) \to \mathbb{R} \]
and we showed that these satisfy the important asymptotic formulae
\begin{equation}
\label{eqn:weyl}
\lim_{k \to \infty} f_k(g) = \Cal(g) 
\end{equation}
on $\text{Diff}_{\mathrm{c}}(\D^2,\omega)$,  and 

\[\lim_{k \to \infty} \mu_k(g) = 0.\]
We called this the ``Calabi property".  Here, $\Cal$ denotes the aforementioned Calabi homomorphism and $\Diff_{\mathrm{c}}$ denotes the group of diffeomorphisms that are the identity near the boundary and that preserve $\omega$, which we note for the reader coincides with the group $\Ham_{\mathrm{c}}$ in the above cases.  We refer the reader to our review in Section~\ref{sec:prelim} for more details about the $\mu_k$ and $f_k$.

The above formulas are kinds of Weyl laws. For specialists, we note that the convergence to zero for the $\mu_k$ is what one would hope for in a Weyl law, since these invariants are defined via mean normalization of Hamiltonians. It is natural to ask what can be said about the subleading asymptotics.   With many seemingly similar kinds of Weyl laws, this tends to be a hard question.  For example, the above Calabi property was inspired by an analogous Weyl law for the related ``ECH spectral invariants" defined in \cite{QECH}, see \cite{CGHR}.  For these spectral invariants,  all that is known is a bound on the growth rate of the subleading asymptotics \cite{CGS} that is likely far from optimal, with the conjectural bound being $O(1)$ \cite{HutchingsConjecture}.

In contrast, it turns out that we are able to say quite a lot about the subleading asymptotics of the $\mu_k$.   To state our result, let $\Ru$ denote the Ruelle invariant from  \cite{Ruelle}  (see also \cite{Gambaudo-Ghys, Ghys_ICM}), which we review in Section \ref{ss:Ruelle}.  We now state a result that is central to our proof of Theorem~\ref{thm:notsimple} and which is also of independent interest.   

\begin{thm}\label{t:O(1)}
If $\psi \in \Diff(\S^2,\omega)$ (resp. $\psi \in \Diff_{\mathrm{c}}(\D^2,\omega) \cap \ker(\Cal)$), then the sequence $\{k\,\mu_k(\psi)\}_{k \in \mathbb{N}}$ 
(resp. $\{k\,f_k(\psi)\}_{k \in \mathbb{N}})$ is bounded.  In fact, if $\psi = \phi^1_H$, where  $H:\D^2 \to \mathbb{R}$  is an autonomous and compactly supported Hamiltonian on the disc with finitely many critical values, then
\begin{equation}
\label{eqn:twoterm} 
\lim_k k\mu_k(\psi) = \lim_k k(f_k(\psi) - \Cal(\psi)) =  \Cal(\psi) - \frac{1}{2} \Ru(\psi).
\end{equation}
\end{thm}

A similar result concerning the subleading asymptotics of the $\mu_k$ in the case of autonomous Hamiltonians on the sphere with finitely many critical values also holds, but for brevity (and because the Ruelle invariant is not defined over the sphere without further choices), we do not state it.

\begin{rem}
In the statement of the above theorem, we are implicitly invoking the fact that we can regard any $\psi \in \Diff_{\mathrm{c}}(\D^2,\omega)$ as a map of the two-sphere by embedding $\D^2$ as a hemisphere and extending by the identity; 
 we will continue to do this throughout this paper.  The invariants $\mu_k$ and $f_k$ can be thought of as invariants of (possibly time-dependent) Hamiltonians as well, by setting $\mu_k(H) := \mu_k(\phi^1_H)$ and $f_k(H) := f_k(\phi^1_H).$  This viewpoint is helpful and adopted in \cite[Sec. 3]{CGHMSS21}, as well as Section \ref{s:subleading} here. 
\end{rem}

In view of Theorem~\ref{t:O(1)} it is natural to ask if \eqref{eqn:twoterm} holds more generally.  For the aforementioned ECH spectral invariants, essentially the same question was asked, under a genericity assumption on the contact form \cite{HutchingsConjecture}.  In the ECH case, simple examples exist, for example the boundary of the round sphere, with no well-defined subleading asymptotic limit at all; in this sense, then, the genericity assumption can not be dropped.  In our case, however, we know of no such analog, and indeed Theorem~\ref{t:O(1)} asserts that in the simplest cases, the subleading asymptotics in fact always recover Ruelle.  We therefore pose as a question the following.

\begin{que}
\label{que:twoterm}
Is it the case that for any $\psi \in \Diff_{\mathrm{c}}(\bD^2,\omega)$, 
\[ \lim_k k\mu_k(\psi) = \lim_k k(f_k(\psi) - \Cal(\psi)) =  \Cal(\psi) - \frac{1}{2} \Ru(\psi)?\]
\end{que}

We emphasize that, in contrast to the ECH case, we are not requiring any genericity in $\psi$ in the above question.  

\begin{rem}
\label{rem:pfh}
 If it was known that homogenized PFH spectral invariants \cite{CGHS21} are quasi-morphisms, it would follow from \cite{Chen2} that they agree with the $\mu_k$ and hence satisfy a two-term Weyl law; the simpler one-term Weyl law is established in \cite{CGPZ, EH}.
\end{rem}



\subsection{Infinitely many extensions of Calabi and the simplicity Conjecture revisited}

Consideration of the asymptotics of the $\mu_k$ also leads to the resolution of an old question about the aforementioned Calabi homomorphism.

\begin{que} [\cite{fathi}]
\label{que:ext}
Does $\Cal:\Diff_{\mathrm{c}}(\mathbb{D}^2,\omega)\to\R$ extend to a group homomorphism $\Homeo_{\mathrm{c}}(\mathbb{D}^2,\omega)\to\R$ ?
\end{que}   

Question~\ref{que:ext} has a long history which is closely connected to the question of whether or not the group $\text{Homeo}_{\mathrm{c}}(\mathbb{D}^2, \omega)$ is simple; see for example \cite[Sec.\ 2.2]{Ghys_ICM}.  It is known that no $C^0$-continuous extension can exist, 
because the kernel of $\Cal$ is  $C^0$-dense.  It was also recently shown that this group is in fact not simple \cite{CGHS20}, resolving the longstanding ``simplicity Conjecture''.  However, the question of whether an extension as a homomorphism exists has remained open.   

One might guess that no such extension exists.  For example, many groups of homeomorphisms satisfy an automatic continuity property, see for example \cite{Mann21}, and as was stated above, it is known that a continuous extension can not exist; see also Remark~\ref{rem:zorn} below.   On the contrary, however, we have the following result.

\begin{thm}
\label{thm:cal}
The Calabi homomorphism admits infinitely many extensions to group homomorphisms $\Homeo_{\mathrm{c}}(\mathbb{D}^2, \omega)\to\R$.
\end{thm}

It follows from Theorem~\ref{thm:cal} that the group $\Homeo_{\mathrm{c}}(\mathbb{D}^2,\omega)$ is not simple.  This gives another proof of the aforementioned ``simplicity Conjecture".  
It should be emphasized that our proof uses the nontrivial construction of the $f_k$ from \cite[Thm. 7.7(iii)]{CGHMSS21} (see Sec \ref{s:qonS} below), so is not self-contained;  
on the other hand, it does give a new proof, deducing nonsimplicity purely algebraically from the existence of 
a geometrically constructed homomorphism out of $\Homeo_{\mathrm{c}}(\mathbb{D}^2,\omega)$.  This kind of argument for proving non-simplicity is much more in line with how non-simplicity is proved for related groups, see the summary in \cite[Sec. 1.1.1]{CGHS20}, so it is natural to hope for a proof like this.
Moreover, this perspective has value in finding new normal subgroups: to keep the introduction focused, we defer the precise statement regarding these subgroups to Section  \ref{sec:newnormal} below.  

\begin{rem}
The homomorphisms we construct in proving Theorem~\ref{thm:cal} are far from canonical.  On the other hand, we will see that our proof does give a natural extension of the Calabi homomorphism to a homomorphism 
$\Homeo_{\mathrm{c}}(\mathbb{D}^2, \omega)\to R'$, where $R'$ is a certain group containing $\mathbb{R}$ as a subgroup; see \eqref{eqn:defns} and \eqref{eqn:exten}.
\end{rem}

\subsection{Simplicity}

 Given Theorem~\ref{thm:notsimple}, it is natural to ask if {\em some} simple non-trivial normal subgroup of $\Homeo_{\mathrm{c}}(\mathbb{D}^2,\omega)$ exists.  After all, there certainly exist groups (e.g. $\mathbb{Z}$) with no simple normal subgroups at all.  

\begin{thm}
\label{thm:comm}
Let $G := \Homeo_{\mathrm{c}}(\Sigma,\omega)$, where $\Sigma$ is some compact surface. The commutator subgroup $[G,G]$ is simple.
\end{thm}

The proof of Theorem~\ref{thm:comm} is completely independent of our other results, and does not use link spectral invariants at all.  In fact, we should note that from a certain point of view, Theorem~\ref{thm:comm} is not too surprising.  Indeed, the commutator subgroup of $[G,G]$ is normal in $G$, so standard arguments as in \cite{fathi}, see in particular the exposition in \cite[Prop. 2.2]{CGHS20}), show that $[G,G]$ is perfect; and, for many transformation groups, perfectness and simplicity are equivalent.  

It would be very interesting to find a geometric characterization of $[G,G]$.  In the diffeomorphism case, Banyaga has shown \cite{Banyaga} that $[G,G]$ is the kernel of Cal.

\subsection{Themes of the proofs and outline of the paper}

A crucial fact for many of our arguments is the following estimate from \cite[Eq. (70)]{CGHMSS21} on the defect of the $f_k$.  (We refer the reader to \ref{sec:prelim-quasis} for preliminaries about quasimorphisms.)

\begin{lem}[\cite{CGHMSS21}, the proof of Thm. 7.6, 7.7 and Eq. (70)]\label{lem:defect}
The $f_k$ and $\mu_k$ are quasimorphisms of defect $\frac{2}{k}$.
\end{lem}

This is a key property that powers many of our arguments and one goal of our paper is to illustrate the usefulness of this fact.
The basic idea is that this defect property allows us to detect interesting normal subgroups and construct interesting homomorphisms; on the other hand, our two-term Weyl law from above allows us to recover the Calabi and Ruelle invariants, which are among the most studied invariants of area-preserving disc maps, 
from the $f_k$, for a wide class of diffeomorphisms.

We put this together as follows.  In our previous work, we studied twist maps with ``infinite Calabi" invariant, defined via the leading asymptotics of the $f_k$, to show that Hameo is proper.  Here, we study twist maps with ``infinite Ruelle invariant," defined via the asymptotics of the $kf_k$, to show nonsimplicity of Hameo.  More precisely, we define a subgroup of elements with $O(1)$ subleading asymptotics and we show that this contains all smooth Hamiltonian diffeomorphisms, but we show that it is proper by constructing a hameomorphism with ``infinite Ruelle invariant;'' 
see Proposition \ref{prop:N-disc}.  

We further comment on the contrast between  ``infinite Ruelle" and ``infinite Calabi" in Remark \ref{rem:fast-vs-toofast}. 

\subsection{Summary of our knowledge of the normal subgroup structure}
It seems to us useful to summarize in one place what is known about the normal subgroup structure for the groups that concern us here, and what remains to be understood.

We start with the case of smooth (i.e. $C^\infty$) diffeomorphisms, established by Banyaga, for the sake of comparison.  We let $G^\infty$ denote either the group $\Diff(\S^2,\omega)$ of smooth diffeomorphisms of $\S^2$ which preserve the area 2-form $\omega$, or the group $\Diff_{\mathrm{c}}(\D^2,\omega)$ of compactly supported smooth diffeomorphisms of $\D^2$ which preserve $\omega$.  As mentioned above, in the case of $\S^2$, we have
\[ [G^\infty,G^\infty] = \Diff(\S^2,\omega),\]
and in the case of $\D^2$ we have
\[ [G^\infty,G^\infty] = \ker(\Cal) \subsetneq \Diff_{\mathrm{c}}(\D^2,\omega).\]
Moreover, $[G^\infty,G^\infty]$ is simple; and, in the disc case, we have
\begin{equation}
  \label{eq:abelianization-smooth}
  G^\infty/[G^\infty, G^\infty]\simeq \R.
\end{equation}

The case of homeomorphisms seems quite different: a striking phenomenon, which seems genuinely new, is a plethora of normal subgroups arising from different geometric considerations.  

To elaborate, we described above the subgroup $\Hameo$, which one can think of as those homeomorphisms that can be said to have Hamiltonians.  There is another normal subgroup $\FHomeo$, containing $\Hameo$, whose precise definition we skip for brevity: one can think of it as the largest normal subgroup for which Hofer's geometry can be defined.  
Buhovsky has recently shown \cite{Bu22} that $\FHomeo$ and $\Hameo$ do not coincide.  As mentioned above, we showed in \cite{CGHS20,CGHMSS21}, resolving in particular the simplicity conjecture, that $\FHomeo$ is proper.   We can therefore summarize the situation regarding these groups, prior to this work, as follows.  Let $G$ denote the group of area and orientation preserving homeomorphisms of $\S^2$ or the group of compactly supported area-preserving homeomorphisms of $\D^2$.

For $\S^2$, we have
\[ [G,G]\subset\Hameo(\S^2,\omega)\subsetneq\FHomeo(\S^2,\omega)\subsetneq G \]
For $\D^2$ we have
\[ [G,G]\subset\ker(\Cal)\subsetneq\Hameo(\D^2,\omega)\subset \FHomeo(\D^2,\omega)\subsetneq G,\]
where here $\Cal$ denotes the extension of the Calabi homomorphism mentioned above that we established in \cite[Thm. 1.4]{CGHMSS21}; one expects the inclusion of $\Hameo$ into $\FHomeo$ to be proper by the arguments in \cite{Bu22}. 

Our work here shows that the left most inclusions are proper, by constructing an explicit normal subgroup, and that $[G,G]$ is simple.   The normal subgroups we construct to show properness, denoted by $N(\S^2)$ and $N(\D^2)$ respectively, do contain $[G,G]$, but we do not know if this inclusion is proper. As a result, for $\S^2$, we have
\[ [G,G]\subset  N(\S^2)  \subsetneq \Hameo(\S^2,\omega)\subsetneq\FHomeo(\S^2,\omega)\subsetneq G \]
For $\D^2$ we have
\[ [G,G]\subset   N(\D^2)  \subsetneq \ker(\Cal)\subsetneq\Hameo(\D^2,\omega)\subset \FHomeo(\D^2,\omega)\subsetneq G,\]

To set the context for describing more normal subgroups, it is natural to wonder if (\ref{eq:abelianization-smooth}) has any counterpart for homeomorphisms.  We know that $G/[G,G]$ contains a subgroup isomorphic to $\R$ and that it therefore has the same cardinality as $\R$, since a continuous function on the reals is determined by its values on the rationals.  However, this is all we currently know about $G/[G,G]$.   
On the other hand, in this paper we find some ``quasimorphism subgroups" that can be assumed to contain any of the above $H$ whose quotients are isomorphic to $\mathbb{R}$, see Section~\ref{s:extension}.

There are additional interesting normal subgroups related to the underlying geometry that are not our focus in the present work.
%
First of all, one can construct normal subgroups via ``fragmentation norms", see \cite{LeRoux10}; it is not currently known how these relate to the normal subgroups above.  One can also find normal subgroups between $\FHomeo$ and $G$ by pulling back from the quotient subgroups corresponding to growth rates of infinite twist maps, see \cite{Pol-Shel21}.  

\medskip

\subsection{Organization of the paper}

The outline of the paper is now as follows.  After reviewing the preliminaries, we start with the computation in the smooth case, proving Theorem~\ref{t:O(1)}; this is the content of Section \ref{s:subleading}.   We then move to the case of hameomorphisms in Section \ref{s:non-simple}: the outcome of the computation from the previous section gives an explicit formula for the subleading asymptotics in the smooth case, and this motivates our definition for a hameomorphism with unbounded subleading asymptotics, see Section \ref{s:non-simple}, which is the key step in proving Theorem~\ref{thm:notsimple}.  Section \ref{s:extension} uses related ideas to extend the Calabi invariant: the idea is that, just as the subleading asymptotics are a suitable replacement for Ruelle, the leading asymptotics allow for infinitely many extensions of Calabi.  Finally, in Section \ref{s:simple}, we prove the simplicity result Theorem~\ref{thm:comm}.

\subsection{Acknowledgements}
DCG thanks Lei Chen, Igor Frenkel, Dan Pomerleano, Christian Rosendal, and Rich Schwartz for helpful and interesting discussions and/or correspondence.   
SS thanks Yannick Bonthonneau for helpful conversations.
IS is grateful to Peter Varju for helpful discussions.  We thank the anonymous referee, and the editor, for their valuable comments.  

DCG also thanks the National Science Foundation for their support under Awards \#1711976 and \#2105471.  VH is partially supported by the Agence Nationale de la Recherche grants ANR-21-CE40-0002 and ANR-CE40-0014.
CM is supported by the Simons Collaboration on Homological Mirror Symmetry. SS is partially supported by the ERC Starting Grant number 851701.



\section{Preliminaries}
\label{sec:prelim}

We begin by reviewing the relevant background material and elaborating on some definitions mentioned in the introduction.

\subsection{The groups, the Calabi homomorphism and the Ruelle invariant}

\subsubsection{Basic notions}
Let $S$ be either the standard $2$-sphere $\S^2=\{(x,y,z)\in\R^3:x^2+y^2+z^2=1\}$ in $\R^3$ or the standard closed $2$-disc $\D^2$ in $\R^2$. We assume that $S$ is endowed with an area form $\omega$. In the case of the disc, unless otherwise stated, all our maps will be assumed compactly supported, i.e. functions on $\D^2$ are assumed to vanish in some neighborhood of the boundary of $\D^2$ and homeomorphisms of $\D^2$ are assumed to coincide with the identity in some neighborhood of the boundary of $\D^2$. 

As mentioned in the introduction, our main characters will be 
\[G=\Homeo_{\mathrm{c}}(S,\omega)\]
the group of (compactly supported in the interior) area-preserving homeomorphisms of $S$
and its smooth counter-part
\[G^{\infty}=\Diff_{\mathrm{c}}(S,\omega)\]
the group of area preserving diffeomorphisms of $S$. When $S=\S^2$ we will drop the subscript `c' from the notation.
The group $G$ is known to be the $C^0$-closure of $G^{\infty}$.
A smooth Hamiltonian $H=(H_t)_{t\in[0,1]}:[0,1]\times S\to \R$ generates an isotopy $(\phi_H^t)_{t\in[0,1]}$ obtained by integrating the time dependent vector field $X_{H_t}$ defined by $\omega(X_{H_t},\cdot)=dH_t$.  
It is known that $G^{\infty}$ coincides with the Hamiltonian group $\Ham_{\mathrm{c}}(S,\omega)$, i.e. that any $\psi\in G^{\infty}$ is of the form $\psi=\phi_H^1$ for some Hamiltonian $H$.

In the disc case, $G^\infty=\Ham_{\mathrm{c}}(\D^2,\omega)$ admits a non-trivial group homomorphism $\Cal:G^{\infty}\to\R$, called the Calabi homomorphism, which we now recall.

\begin{defn}\label{def:Calabi}
Let $\psi\in \Diff_{\mathrm{c}}(\D^2,\omega)$. Since $G^\infty=\Ham_{\mathrm{c}}(\D^2,\omega)$, the diffeomorphism $\psi$ is the time-one map of a Hamiltonian $H$, i.e. $\psi=\phi_H^1$. The quantity
\[\Cal(\psi)=\int_0^1\int_{\D^2}H\,\omega\,dt\]
turns out to be independent of the choice of Hamiltonian $H$ and is called the Calabi invariant of $\psi$. This defines a map $\Cal: G^\infty\to\R$ which is a group homomorphism \cite{calabi} (see also \cite{mcduff-salamon}).
\end{defn}

As mentioned in the introduction, we can think of the Calabi homomorphism as measuring the ``average rotation'' of the map, see \cite{fathi-calabi,Gambaudo-Ghys}.

\subsubsection{Some normal subgroups of $G$}
We are interested in this work in a particular normal subgroup of $G$.  The key definition is as follows. As above, we denote by $S$ a surface which is either $\S^2$ or $\D^2$. 

\begin{defn}[Oh-M\"uller \cite{muller-oh}]\label{def:hameo} A homeomorphism $\psi\in G$ is called a hameomorphism (or sometimes a strong Hamiltonian homeomorphism) if there exist a compact subset $K\subset S$, a sequence of Hamiltonians $H_i:[0,1]\times S\to\R$, $i\in\mathbb N$, supported in $K$ and an isotopy $(\psi^t)_{t\in[0,1]}$ with $\psi^0=\id$ and $\psi^1=\psi$, such that
  \begin{enumerate}[(i)]
  \item $\phi_{H_i}^t$ converges to $\psi^t$ in the $C^0$ topology and uniformly in $t\in[0,1]$,
  \item $H_i$ is a Cauchy sequence with respect to the Hofer norm $\|\cdot\|$. 
  \end{enumerate}
The set of all hameomorphisms is denoted $\Hameo(S,\omega)$. It is proved that $\Hameo(S,\omega) \neq G$ \cite{CGHS20}.
\end{defn}

\begin{rem}\label{rem:different-hameos} Several variants of the above definition may be found in the literature. In particular, one sometimes replaces the convergence with respect to the Hofer norm $\|\cdot\|$ with uniform convergence. However, it was proved by M\"uller \cite{Muller} that this change in the definition gives rise to the same group of hameomorphisms.

  In \cite[Def. 2.1]{CGHMSS21}, we used the following weaker variant. We called $\psi$ a hameomorphism if there exist a compact set $K$ and a sequence of Hamiltonians $H_i$, supported in $K$, such that the time-1 maps $\phi_{H_i}^1$ converge to $\psi$ and $H_i$ is Cauchy with respect to the Hofer norm. This weaker notion gives rise to another normal subgroup of $G$, which we will denote $\Hameo'$ in this remark. We clearly have the inclusion $\Hameo\subset\Hameo'$, but we do not know whether equality holds.

  Our reason to change from one notion to another is to have stronger statements. Indeed, in \cite[Thm. 1.4]{CGHMSS21} (see also the discussion in Theorem~\ref{th:extension-calabi} below), we extended the Calabi homomorphism to $\Hameo'$ which is a priori a stronger result than just extending to $\Hameo$. Here, we find a normal subgroup of $G$ which is strictly smaller than $\Hameo$ (resp. $\ker(\Cal)$ in $\Hameo$); this is a priori a stronger statement than finding a subgroup in $\Hameo'$ (resp. $\ker(\Cal)$ in $\Hameo'$).  
\end{rem}

We can use the Calabi homomorphism from above to get some additional subgroups, as the following shows.

\begin{thm}[\cite{CGHMSS21}, Theorem 1.4]\label{th:extension-calabi} The Calabi homomorphism on $G^\infty$ extends canonically to a group homomorphism  $\Hameo(\D^2,\omega)\to\R$. 
  Moreover, for any $\psi\in \Hameo(\D^2,\omega)$ and any sequence $H_i$ as in Definition \ref{def:hameo} the extension of the Calabi homomorphism satisfies
  \[\Cal(\psi)=\lim_{i\to\infty} \Cal(\phi_{H_i}^1).\]
\end{thm}

This gives another normal subgroup of $G$ in the case of the disc, namely the kernel of $\Cal:\Hameo(\D^2,\omega)\to\R$.

\subsubsection{The Ruelle invariant}\label{ss:Ruelle}
We now recall the construction of the Ruelle quasi-morphism, following \cite{Gambaudo-Ghys}. Recall that  $G^{\infty} := \Diff_{\mathrm{c}}(\bD^2,\omega)$ denotes the group of compactly-supported area-preserving diffeomorphisms of the 2-disc. We fix a trivialization
\begin{equation} \label{eqn:trivialise_TD}
T\bD^2 \cong \bD^2 \times \mathbb{R}^2
\end{equation}
(which is unique up to homotopy).  The group $G^{\infty}$ is contractible, so if $g\in G^{\infty}$ we may pick an isotopy $\{g_t\}$ from $\id$ to $g$, again unique up to homotopy.  For a point $z \in \bD^2$, let
\[
v_t(z) \in \bR^2\backslash \{0\}
\]
denote the first column of $dg_t(z) \in \operatorname{SL}(2,\bR)$ expressed
in the trivialisation \eqref{eqn:trivialise_TD}, and 
\[
\mathrm{Ang}_g(z) \in \mathbb{R}
\]
the variation in the angle of $v_t(z)$, measured with respect to a fixed direction (say the $x$-axis) and integrated over $0\leq t \leq 1$.  The uniqueness of the choice of $\{g_t\}$ up to homotopy shows this does not depend on the choice of isotopy from $g$ to $\id$. The function $z\mapsto \mathrm{Ang}_g(z)$ is smooth and so integrable. 
Setting 
\[
r(g) := \int_{\bD^2} \mathrm{Ang}_g(z) \, \omega
\]
we obtain the \emph{Ruelle invariant}
\[
\Ru(g) := \lim_{p \to \infty} r(g^p)/p
.\]
This is a non-trivial homogeneous quasi-morphism on $G^{\infty}$ (and on the kernel of the Calabi homomorphism).

Gambaudo and Ghys \cite[Proposition 2.9]{Gambaudo-Ghys} give a formula for the Ruelle invariant in the special case of an autonomous Hamiltonian flow of a function $H \in C^{\infty}_{\mathrm{c}}(\bD^2)$ with finitely many critical values. Suppose  $\xi \in \bR$ is a regular value of $H$, so $H^{-1}(\xi)$ is a finite disjoint union of circles. Each such circle $C$ bounds a disc in $\bD^2$, and we associate the sign $+1$, respectively $-1$, to $C \subset H^{-1}(\xi)$ depending on whether $H$ increases,  respectively decreases, as one crosses from the exterior to the interior region. 

Then 
\begin{align}
\Ru(H) := \Ru(\phi_H^1) = \int_{\bR} n_H(\xi) d\xi \label{eq:RuelleFormula}
\end{align}
where the integer $n_H(\xi) \in \bZ$ is the signed sum of values $\pm 1$ over the connected components $C$ of $H^{-1}(\xi)$. 

Specialising further to the case of a smooth function $H\in C^{\infty}_{\mathrm{c}}(\bD^2)$ which is Morse with critical points $p_i$, this simplifies to (\cite[Section 2.4]{Ghys_ICM}):
\begin{equation} \label{eqn:Ruelle_Morse_case}
\Ru(H) = \sum_i (-1)^{\mathrm{ind}(p_i)}\, H(p_i)
\end{equation}
where $\mathrm{ind}(p_i)$ is the Morse index of $p_i$.

\subsection{Monotone links, spectral invariants and quasimorphisms}

The material for this section was developed in \cite{CGHMSS21}. We refer the reader to this paper for further details. 

\subsubsection{Monotone links and their spectral invariants}
We call a Lagrangian link (or Lagrangian configuration) any subset of the form $\underline{L}=L_1\cup \dots\cup L_k$ where the $L_i$'s are pairwise disjoint smooth simple closed curves in $\S^2$, see Figure \ref{fig:ex-link}. A Lagrangian link is called {\it monotone} if the connected components of its complement all have the same area $\frac{\mathrm{area}(\S^2)}{k+1}$.  

\begin{figure}[h!]
\centering
\includegraphics[width=0.7\columnwidth]{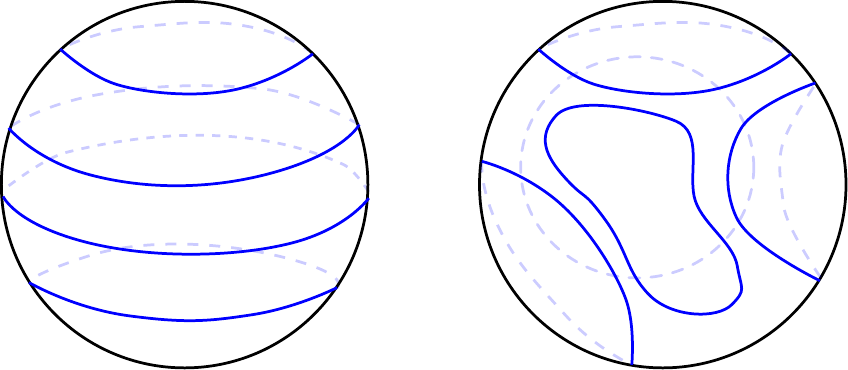}
\caption{Two examples of Lagrangian links on $\S^2$ with respectively $k=4$ and $k=5$ components.}
\label{fig:ex-link}
\end{figure}

\begin{rem}
  In \cite[Def. 1.12]{CGHMSS21}, we introduced a more general notion of $\eta$-monotonicity, where $\eta$ is a non-negative real parameter. We will not need this more general notion in the present paper. What we call monotonicity here corresponds to $0$-monotonicity.
\end{rem}

Let $\ul{L}$ be a monotone link  with $k$ components. We can take the product of the components to form the associated connected submanifold $\Sym(\ul{L})$ inside the $k$-fold symmetric product  $\Sym^k(\S^2):=(\S^2)^k/\Sym_k$, where $\Sym_k$ is the permutation group permuting the factors of $(\S^2)^k$.
The symplectic form $\omega$ on $\S^2$ induces a singular symplectic form on $\Sym^k(\S^2)$ whose singular locus is away from $\Sym(\ul{L})$ and makes $\Sym(\ul{L})$ a Lagrangian submanifold.
After smoothing the symplectic form near the singular locus, the Lagrangian Floer cohomology of $\Sym(\ul{L})$ with itself is well-defined and non-zero \cite[Lem.\ 6.10]{CGHMSS21}. It enables us to define the link spectral invariants as follow.

Given a Hamiltonian function $H: [0,1] \times \S^2 \to \mathbb{R}$, we define $\Sym(H): [0,1] \times \Sym^k(\S^2) \to \mathbb{R}$ to be $\Sym(H)_t([x_1,\dots,x_k]):=\sum_{i=1}^k H_t(x_i)$.
The Lagrangian link spectral invariant $c_{\ul{L}}(H)$ is defined to be $\frac{1}{k}c_{\Sym(\ul{L})}(\Sym(H))$, where $c_{\Sym(\ul{L})}(\Sym(H))$ is the Lagrangian spectral invariant of $\Sym(H)$ with respect to the Lagrangian submanifold $\Sym(\ul{L})$ \cite[Equation (54)]{CGHMSS21}.
We have shown in \cite[Thm 1.13, Lem. 6.16, 6.17]{CGHMSS21} that it is well-defined and independent from the choice of smoothing of the symplectic form as long as the smoothing is sufficiently local.
For a Hamiltonian diffeomorphism $\psi \in \Diff(\S^2,\omega)\ (=\Ham(\S^2,\omega) )$ and a mean-normalized generating Hamiltonian $H$ (i.e. $\int_{\S^2} H_t \omega dt=0$ for all $t \in [0,1]$, and $\psi=\phi^1_H$), we have shown in \cite[Thm 1.13, Lem. 6.16, 6.17]{CGHMSS21} that $c_{\ul{L}}(\psi):=c_{\ul{L}}(H)$ is well-defined and independent of the choice of $H$.
\begin{rem}
In the specific case of a link $\ul L$ by parallel circles, similar invariants were previously constructed by Polterovich and Shelukhin \cite{Pol-Shel21} using orbifold Floer cohomology \cite{FOOObulk}, \cite{Cho-Poddar} and the computational techniques in \cite{MS19}.
\end{rem}

\subsubsection{Quasimorphisms for diffeomorphism groups}\label{sec:prelim-quasis}

Recall that a quasimorphism on a group $\Gamma$ is a map $f:\Gamma\to\R$ for which there exists a constant $D>0$ such that for any $a,b\in\Gamma$,
\[|f(ab)-f(a)-f(b)|\leq D\]
The constant $D$ is called a defect of $f$. A quasimorphism $f$ is said to be homogeneous if it satisfies $f(a^k)=kf(a)$ for any $k\in\bZ$ and $a\in \Gamma$.

The spectral invariant $c_{\underline{L}}$ may be used to construct quasimorphisms on $G^{\infty}=\Diff(\S^2,\omega)$. This was proved in \cite[Thm 7.6]{CGHMSS21}, inspired by an older (and famous) construction of Entov and Polterovich \cite{Entov-Polterovich}.

Let $\underline L$ be a monotone Lagrangian link with $k$ components and $\phi\in\Diff(\S^2,\omega)$ and let us introduce the homogenized spectral invariant
\begin{equation*}
  \mu_k(H)=\lim_{n\to\infty}\frac1n c_{\underline L}(H^{\sharp n}),
\end{equation*}
for any Hamiltonian $H$ on $\S^2$; these are the $\mu_k$ mentioned in the introduction.  The above limit does not depend on the choice of the link $ \underline L$; see Theorem \ref{thm:mu-k} below.    Here the notation $H^{\sharp n}$ means the $n$-times composition of $H$, where the composition of Hamiltonians is defined by $(H\sharp K)_t(x)=H_t(x)+K_t\circ(\phi_H^t)^{-1}(x)$. It is well known that $H\sharp K$ generates $\phi_H^t\circ\phi_K^t$, thus $H^{\sharp n}$ generates the isotopy $(\phi_H^t)^n$.

Note that $\mu_k$ has a shift property (see \cite[Thm 1.13]{CGHMSS21}), namely for any Hamiltonian $H$ and any constant $c\in\R$, we have
\begin{equation}
  \label{eq:shift}
  \mu_k(H+c)=\mu_k(H)+c.
\end{equation}
As above, we obtain invariants associated to elements of $\Diff(\S^2,\omega)$ (still denoted $\mu_k$) by:
\begin{equation}\label{eq:prelim-H-phi}
  \mu_k(\varphi)=\mu_k(H),
\end{equation}
for any mean-normalized Hamiltonian $H$ such that $\phi_H^1=\varphi$. This does not depend on the choice of $H$, see \cite[Thm. 1.13, Lem 6.17]{CGHMSS21}. 

\begin{thm}[\cite{CGHMSS21} Thm.\ 7.6, Thm.\ 7.7]\label{thm:mu-k} For fixed $k$, the map $\mu_k:\Diff(\S^2,\omega)\to\R$ does not depend on the choice of Lagrangian link $\underline L$.
Moreover the following properties hold
  \begin{enumerate}
  \item (Hofer continuity and monotonocity) For all Hamiltonians $H, K$, 
    \[\int_0^1\min_{x\in\S^2}(H_t(x)-K_t(x))dt\leq \mu_k(H)-\mu_k(K)\leq \int_0^1\max_{x\in\S^2}(H_t(x)-K_t(x))dt.\]
  \item (Lagrangian control) Let $H$ be a 
    Hamiltonian and $\underline{L}=L_1\cup\dots\cup L_k$ be a Lagrangian link, such that for all $i=1,\dots, k$ the restriction of $H$ to $L_i$ is a function of $t$ denoted $c_i$.  Then,
    \[\mu_k(H)=\frac1k\sum_{i=1}^k\int_0^1c_i(t)dt.\]
 \item (Quasimorphism) The map $\mu_k$ is a homogeneous quasimorphism of defect $\frac2k$.  
    \end{enumerate}
\end{thm}

The first item implies that the quasimorphisms $\mu_k:\Diff(\S^2,\omega)\to\R$ are Lipschitz continuous with respect to Hofer distance $d_H$ on $\Diff(\S^2,\omega)$ defined by
  \begin{equation}
    \label{eq:def:hofer-distance}
    d_H(\varphi, \psi):=\inf_{\varphi=\phi_H^1,\psi=\phi_K^1}\|H-K\|,
\end{equation}
where the norm is given by $\|H\|:=\int_0^1(\max_{\S^2}H_t-\min_{\S^2}H_t)dt$ (See e.g. \cite{Polterovich2001} for an introduction to Hofer's distance).
A consequence of item 2, proved in \cite[Thm 7.7(ii)]{CGHMSS21}, is that the quasimorphisms $\mu_k$ are linearly independent. In the case $k=1$, we recover the Entov-Polterovich quasimorphism  \cite{Entov-Polterovich}. 
As a consequence of the first and second items, for any Hamiltonian $H$, we have
\begin{equation}
\label{eq:strongerLagcontrol}
\frac1k\sum_{i=1}^k\int_0^1\min_{x \in L_i}H_t(x)dt \le \mu_k(H) \le \frac1k\sum_{i=1}^k\int_0^1\max_{x \in L_i}H_t(x)dt,
\end{equation}
which will be useful to us later.

\subsubsection{Quasimorphisms on the sphere}\label{s:qonS}
We now introduce quasimorphisms on the sphere.  Denote
\[f_k:=\mu_k-\mu_1.\]
By (\ref{eq:prelim-H-phi}) and the shift property (\ref{eq:shift}), we have 
\begin{equation*}
  f_k(H)=f_k(\phi_H^1)
\end{equation*}
for all Hamiltonians $H$ (not only for mean-normalized ones).
The $f_k$ give quasimorphisms on $\Diff(\S^2,\omega)$ which have similar properties to the $\mu_k$. Our motivation for introducing them is their $C^0$-continuity, which is not satisfied by the $\mu_k$. We collect in the next theorem their useful properties.

\begin{thm}[\cite{CGHMSS21}, Thm. 7.7 (iii), Thm. 7.6 (support control)]\label{th:C0continuity-fk}
  \begin{enumerate}
  \item ($C^0$-continuity) For all $k\geq 1$, the quasimorphism $f_k$ is continuous with respect to $C^0$ topology and extends continuously to $\Homeo(\S^2,\omega)$.
  \item (Support control) For all $k\geq 1$ and $\phi\in \Homeo(\S^2,\omega)$ whose support is included in a disc of area $\leq \frac1{k+1}$, we have
    $f_k(\phi)=0$.  
  \end{enumerate}
\end{thm}

\begin{rem}\label{rem:difference-of-qm} In fact, for any positive integers $k, k'$, the difference $\mu_{k'}-\mu_k$ extends continuously to a quasimorphism on $\Homeo(\S^2,\omega)$. Its defect is bounded above by the sum of the defects of $\mu_k$ and $\mu_{k'}$, i.e. by $\frac2{k}+\frac2{k'}$.  In particular, $f_k$ has defect $\frac2{k}+2$. 
\end{rem}

\subsubsection{Inducing quasimorphisms on the disc}\label{sec:qm-on-disc}

Let $\iota:\D^2\to\S^2$ be a smooth symplectic embedding which identifies the disc $\D^2$ with the northern (or southern) hemisphere. Then we have an inclusion  $\Homeo_{\mathrm{c}}(\D^2,\omega)\subset \Homeo(\S^2,\omega)$ and the maps $f_k$ induce by restriction quasimorphisms on $\Homeo_{\mathrm{c}}(\D^2,\omega)$.

Let $H$ be a Hamiltonian which is compactly supported in the disc. Then the Lagrangian control property yields $\mu_1(H)=0$ hence
\begin{equation}\label{eq:fkmuk}
f_k(\phi_H^1)=\mu_k(H)=\mu_k(\phi_H^1)+\int_0^1\int_{\S^2}H\,\omega\,dt.
\end{equation}
and in particular we obtain the following strengthening of the bound on the defect in Remark~\ref{rem:difference-of-qm} (which we already stated in Lemma \ref{lem:defect}).

\begin{lem}\label{l:quasi2k}
  The $f_k$ restricted to the disc are quasimorphisms with defect $2/k$.
\end{lem}

Using the Lagrangian control property and a Lagrangian link $\underline L$ consisting of horizontal circles $L_i=\{(x,y,z)\in\S^2\,|\,z=-1 + 2\frac{i}{k+1}\}$, $i=1,\dots, k$, we can compute $f_k$ explicitly for Hamiltonians that only depend on the variable $z$, namely: 
\begin{equation}\label{eq:lag-control-fk}
f_k(\varphi^1_H) = \frac{1}{k} \sum_{i=1}^k H(-1 + 2\tfrac{i}{k+1}).
\end{equation}
This formula will be used in some subsequent sections.



\section{The subleading asymptotics and the Ruelle invariant}\label{s:subleading}

In this section, we first show that the spectral invariants $\{\mu_k\}$ have O(1) subleading asymptotics, and then compute those asymptotics exactly in the case of autonomous disc maps with finitely many critical values.

\subsection{$O(1)$ subleading asymptotics}
The proof that the spectral invariants $\{\mu_k\}$ have $O(1)$ subleading asymptotics in the smooth case is an almost immediate consequence of the key inequality
\begin{equation}\label{eqn:key-defect}
|\mu_k(\psi_0 \psi_1)-\mu_k(\psi_0)-\mu_k(\psi_1)| \le \frac{2}{k} 
\end{equation}
from Lemma~\ref{lem:defect}.

\begin{thm}\label{t:muO(1)}
For any $\psi \in \Diff(\S^2,\omega)$, the sequence $\{k\,\mu_k(\psi)\}_{k \in \mathbb{N}}$ 
is bounded.
\end{thm}

\begin{proof}
Let $G_{O(1)}:=\{\psi \in \Diff(\S^2,\omega)| k\mu_k(\psi)=O(1)\}$. Equation \eqref{eqn:key-defect} shows both that if $\psi_0,\psi_1 \in G_{O(1)}$, then so is the product $\psi_0 \psi_1$, and also that  $\psi \in G_{O(1)}$ if and only if $\psi^{-1} \in G_{O(1)}$.
Therefore, $G_{O(1)}$ is a subgroup of $\Diff(\S^2,\omega)$.
Since $\mu_k(\psi)$ is invariant under conjugating $\psi$ by elements in $\Diff(\S^2,\omega)$,  $G_{O(1)}$ is a normal subgroup.

Since $\Diff(\S^2,\omega)$ is simple, to show $G_{O(1)} = \Diff(\S^2,\omega)$ it therefore suffices to show that $G_{O(1)}$ contains a single non-identity element.  Let $H$ be the height function (projection to $z$ co-ordinate) of $\S^2 \subset \mathbb{R}^3$.
Let $\ul{L}_k$ be the $k$-component monotone link all of whose components are level sets of $H$. 
By the Lagrangian control property, we have $\mu_k(H)=0$. 
Since $H$ is mean-normalized, we have $\mu_k(\phi_H^1)=0$, but $\phi_H^1$ is not the identity element in $\Diff(\S^2,\omega)$. The result follows. 
\end{proof}

By restricting the $\{\mu_k\}$ to Hamiltonians on $\S^2$ supported (for instance) in a hemisphere, we immediately obtain:

\begin{cor}\label{corol:O(1)}
Let $\bD^2$ be a disc in $\S^2$ with area at most half of that of $\S^2$.
For any $\psi \in \Diff_{\mathrm{c}}(\bD^2,\omega)$,  the sequence $\{k \cdot (f_k(\psi) -\Cal(\psi))\}_{k \in \mathbb{N}}$ is bounded.
\end{cor}

\begin{proof}
 It follows from 
\eqref{eq:fkmuk}
 and  Theorem \ref{t:muO(1)}.
\end{proof}

\subsection{Autonomous Hamiltonians}

For general smooth Hamiltonian diffeomorphisms on the disc, we know from above that $k\cdot (f_k(\psi) - \Cal(\psi))$ is bounded as $k\to \infty$, but not that this sequence has a well-defined limit.  For autonomous maps with finitely many critical \emph{values}, the limit does exist, and is determined by the classical Ruelle invariant from Section \ref{ss:Ruelle}: showing this is the aim of this section. 

The main result is the following.
\begin{thm}\label{t:ruelle}
Let $H:(\bD^2,\omega) \to \mathbb{R}$ be a compactly supported autonomous Hamiltonian with finitely many critical  values. Then, 
\begin{align}
\lim_{k \to \infty} \left( k\mu_{k}(H)-(k+1) \Cal(H)\right) =  -\frac{1}{2}\Ru(H). \label{eq:formula}
\end{align}
\end{thm}

Theorem \ref{t:O(1)} directly follows from  Theorem \ref{t:muO(1)}, Corollary \ref{corol:O(1)} and Theorem \ref{t:ruelle}.
The proof of Theorem \ref{t:ruelle} is independent of the rest of the paper and reader might want to skip it in a first reading.

\begin{rem}
The coefficient $k+1$ of $\Cal(H)$ is the reciprocal of the monotonicity constant of a $k$-component link $\ul{L}_k$ (i.e. the area of a connected component of $\S^2 \setminus \ul{L}_k$).
\end{rem}

The proof will use monotone Lagrangian links $\ul{L}_k$ `most' of whose connected components are contained in level sets of $H$.
In order to describe these links, we need the notion of the Reeb graph.

Let $H:(\S^2,\omega) \to \mathbb{R}$ be an autonomous Hamiltonian with finitely many critical values.
We define an equivalence relation $\sim$ on $\S^2$ via  $x \sim y$ if and only if they lie in the same connected component of a level set of $H$.
Let $G:=\S^2/\sim$ be the Reeb graph of $H$ equipped with the quotient topology, and $R: \S^2 \to G$ be the associated quotient map. There is a uniquely defined continuous function $H_{G}:G \to \mathbb{R}$ such that $H=H_{G} \circ R$.

\begin{lem}\label{l:tree}
The space $G$ is homeomorphic to a finite tree.
\end{lem}

\begin{proof}
We are going to describe a finite graph structure on $G$ and then we will show that it is a tree.

We define the set of vertices of $G$ to be the  $H_G$-preimage of the set of critical values of $H$.
We want to show that the set of vertices is finite.
It suffices to show that for every critical value $c$ of $H$, $H^{-1}(c)$ has only finitely many connected components.
Let $J_1 \supset J_2 \supset \dots$ be a nested sequence of open intervals such that $\cap_n J_n=\{c\}$.
By possibly passing to a subsequence, we can assume that $c$  is the only critical value of $H$ in $\overline{J_1}$. 
Under this assumption, the number of connected components of $H^{-1}(J_n)$ equals to the number of connected components of $H^{-1}(\overline{J_n})$ and it is a finite number $k_c$ that is independent of $n$.
Recall that the intersection of a nested sequence of connected compact sets is connected.
Therefore, $H^{-1}(c)$ also has $k_c$ connected components.

We define the complement of the vertices of $G$ to be the open edges of $G$.
We need to show that the complement of the vertices has finitely many connected components and every connected component is homeomorphic to $(0,1)$.
But it follows easily from the fact that for any two consecutive critical values $c_0<c_1$ of $H$, $H|_{H^{-1}((c_0,c_1))}: H^{-1}((c_0,c_1)) \to (c_0,c_1)$ is submersive and hence is a fibre bundle.


Finally, we want to show that $G$ is a tree.
The second paragraph above implies that if $v$ is a vertex of $G$ and $V \subset G$ is a small connected open neighborhood of $v$, then $R^{-1}(V)$ is open and connected, and hence path-connected.
If $G$ were not a tree, then the aforementioned path-connectedness would enable us to find a smooth circle $C \subset \S^2$ and a point $q \in C$ such that $R(C)$ is a non-trivial cycle in $G$, $R(q)$ is on an open  edge of $G$ and $R^{-1}(R(q)) \pitchfork C=\{q\}$.
This would imply that the intersection pairing $[R^{-1}(R(q))] \cdot [C]= \pm 1 \neq 0$, contradicting the fact that $H_1(\S^2;\bZ)=0$. 
\end{proof}

We record here a useful consequence of the argument in Lemma \ref{l:tree} here:

\begin{lem}\label{l:connected}
For any connected open set $U \subset G$, the set $R^{-1}(U) \subset \S^2$ is connected.
\end{lem}

\begin{proof}
We use the sets $J_n$ in the proof of Lemma \ref{l:tree}.
When $V \subset G$ is a connected component of $H_G^{-1}(J_n)$, $R^{-1}(V)$ is one of the connected components of $H^{-1}(J_n)$.
Since $R^{-1}(V)$ is also open, it is path-connected.

On the other hand, when $V\subset G$ is a connected open set not containing any vertex of $G$, $R^{-1}(V)$ is clearly also path connected. 

For a general connected open set $U \subset G$, we can write it as a union of open sets $U = \cup_{\alpha} V_{\alpha}$, each of the $V_{\alpha}$'s is of one of the two types above.
Since every $R^{-1}(V_{\alpha})$ is path-connected and $U$ is path-connected, we conclude that $R^{-1}(U)$ is path-connected and hence connected.
\end{proof}

Let $t$ be the number of vertices of $G$ and enumerate the vertices $v_1,\dots,v_t$. 
Let $\mu$ be the Borel measure on $G$ such that for every open set $U \subset G$, we define 
\[
\mu(U):=\int_{R^{-1}(U)} \omega.
\]
For $i=1,\dots,t$, let 
\[
m_i:=\mu(v_i) \in [0,1].
\]
Note that for any $x \in G \setminus \{v_1,\dots,v_t\}$, we have $\mu(x)=0$. 
Many of the components of our desired links $\ul{L}_k$ will be of the form $R^{-1}(x)$ for some $x \in G \setminus \{v_1,\dots,v_t\}$. 
But when $m_i>0$ and $k$ is large, we also need to put many components near $R^{-1}(v_i)$ for $\ul{L}_k$ to be monotone.
There are connected closed subsets in $\S^2$ with positive measure but empty interior so we have to be particularly careful near $R^{-1}(v_i)$ when we construct $\ul{L}_k$. To construct $\ul{L}_k$, we need to first explain a decomposition of $G$ using $\mu$.

For the following decomposition of $G$, we assume that $k \in \mathbb{N}$ satisfies $\frac{3}{k+1} < \mu(e)$  for every open edge $e \in G$.
For $i=1,\dots,t$, let $\{U_{i,j}\}_{j=1}^{s_i}$ be the connected components of $G \setminus \{v_i\}$, where $s_i:=\val(v_i)$ is the valency of $v_i$, which equals to the number of connected components of $G \setminus \{v_i\}$.
Denote $\mu(U_{i,j})$ by $a_{i,j}$ so we have $m_i+\sum_{j=1}^{s_i} a_{i,j}=1$.
By our assumption on $k$, we have $a_{i,j}>\frac{3}{k+1}$ for all $i,j$.
Let $r_{i,j} \in (0,\frac{1}{k+1}]$ be the unique number such that $a_{i,j}-r_{i,j}$ is an integer multiple of $\frac{1}{k+1}$.

For any $v_i$ and any $j=1,\dots,s_i$, we define $x_{i,j} \in U_{i,j}$ to be the unique point such that $x_{i,j}$ is on an edge adjacent to $v_i$
and the open interval between $v_i$ and $x_{i,j}$ has $\mu$-measure $r_{i,j}$.
The existence of $x_{i,j}$ is guaranteed by the assumption $\frac{3}{k+1} < \mu(e)$, whilst its  uniqueness comes from the fact that $G$ is a tree (so there is a bijective correspondence between edges adjacent to $v_i$ and connected components $U_{i,j}$). 
Moreover, again by the assumption $\frac{3}{k+1} < \mu(e)$, we have that $x_{i,j} \neq x_{i',j'}$ unless $i=i'$ and $j=j'$. 
Most importantly, by our choice of $r_{i,j}$, each connected component of $G \setminus \{x_{i,j}\}_{i,j}$ has $\mu$-measure being an integer multiple of $\frac{1}{k+1}$.
Denote the component of $G \setminus \{x_{i,j}\}_{i,j}$ containing $v_i$ by $V_i$.
Let $S_{k,i}:=\lfloor (k+1)m_i \rfloor$.
By construction, we have $\mu(V_i)=m_i+ \sum_{j=1}^{s_i} r_{i,j} \in (m_i, m_i+ \frac{\val(v_i)}{k+1}]$ so $(k+1)\mu(V_i) $ is an integer in the interval $(S_{k,i}, S_{k,i}+\val(v_i)]$.

As a summary, the graph $G$ is partitioned into the $V_i$ and segments in the edges, all of whose $\mu$-measures
are integer multiples of $\frac{1}{k+1}$. Moreover, the
 $\mu$-measure of each $V_i$ exceeds $m_i$ by less than $\frac{\val(v_i)}{k+1}$.

\begin{construction}\label{c:linkExist}
We consider a monotone Lagrangan link $\ul{L}_k$ comprising of the following $3$-types of circles, which we call them type $T_1$, type $T_{2,i}$ ($i=1,\dots,t$) and type $T_{3,i}$ ($i=1,\dots,t$), respectively (see Figure \ref{fig:link}). 

{\bf Type $T_1$ circles:}  any component of $G \setminus \{x_{i,j}\}_{i,j}$ not containing any $v_i$ is an interval. We can subdivide that interval so that each sub-interval has $\mu$-measure $\frac{1}{k+1}$.
Let $X$ be the union of $\{x_{i,j}\}_{i,j}$ and the additional points we added to subdivide the intervals.
Then $R^{-1}(X)$ gives us 
\[|X| =k-\sum_{i=1}^t ((k+1)\mu(V_i) -1)\]
 circles in $\S^2$. 
For each $i=1,\dots,t$, we remove 
\[S_{k,i}+\val(v_i)-(k+1)\mu(V_i)\] many circles from $R^{-1}(X)$ that are closest to $R^{-1}(V_i)$.
The remaining  circles, \[k-\sum_{i=1}^t (S_{k,i}+\val(v_i) -1)\] many, are the type $T_1$ circles of $\ul{L}_k$.








\begin{figure}[h!]
\centering
\includegraphics{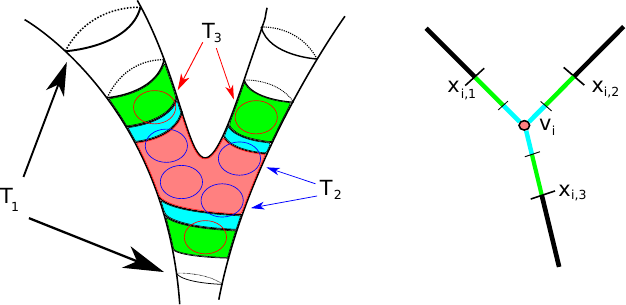}
\caption{On the left: $R^{-1}(J_i)$ (the union of pink and light blue regions) contains $R^{-1}(v_i)$ (pink region) and type $T_{2,i}$ circles (blue). $R^{-1}(V_i)$ (the union of pink, light blue and light green regions) contains both type $T_{3,i}$ circles (red) and type $T_{2,i}$ circles. Type $T_1$ circles (black) are level sets outside the interior of $R^{-1}(V_i)$. On the right: we indicate a neighborhood of the vertex $v_i$ in $G$, coloured to indicate  the images of the respective regions on the left.}
\label{fig:link}
\end{figure}




{\bf Type $T_{2,i}$ circles:} 
Let $J_i \subset V_i$  be a connected open set containing $v_i$ such that $\max_{J_i} H_G - \min_{J_i} H_G < \frac{1}{k^2}$.
By Lemma \ref{l:connected}, we know that $R^{-1}(J_i)$ is a connected open set whose closure has smooth boundary.
 Since the area of $R^{-1}(J_i)$ is $\mu(J_i)>m_i \ge \frac{S_{k,i}}{k+1}$, we can choose $S_{k,i}$ pairwise disjoint discs  of area $\frac{1}{k+1}$ in $R^{-1}(J_i)$.
The boundary of these discs are the type $T_{2,i}$ circles of $\ul{L}_k$.


{\bf Type $T_{3,i}$ circles:}  
The complement of the associated $S_{k,i}$ disjoint closed discs in $R^{-1}(V_i)$ is also a connected open subset of $\S^2$ with area $\mu(V_i)-\frac{S_{k,i}}{k+1}$.
Therefore, we can put an additional set of 
\[(k+1)\mu(V_i) -S_{k,i}-1\] circles, each bounding a disc of area $\frac{1}{k+1}$ again, to obtain a further collection of circles in $R^{-1}(V_i)$.
Adding the 
\[S_{k,i}+\val(v_i)-(k+1)\mu(V_i)\]
many circles that are in $R^{-1}(X)$ but not of type $T_1$, we 
get a collection of 
\[
\left( (k+1)\mu(V_i) -S_{k,i}-1 \right)+ \left( S_{k,i}+\val(v_i)-(k+1)\mu(V_i) \right)=\val(v_i)-1\]
 circles, which are called
type $T_{3,i}$ circles of $\ul{L}_k$.

We use the notation $\ul{L}_{k,j}$ for $j \in T_1$ (resp. $j \in T_{2,i}$, $j \in T_{3,i}$) to refer to a connected component of $\ul{L}_k$ of type $T_1$ (resp. $T_{2,i}$, $T_{3,i}$).
\end{construction}

First note that $\ul{L}_k$  is indeed a $k$-component monotone link, because all the components of its complement  have area $\frac{1}{k+1}$. On top of this, the three types of circles above have the following features respectively:
\begin{enumerate}
\item $H$ is locally constant on type $T_1$ circles.
\item The value of $H$ on type $T_{2,i}$ circles is $m_i$ up to an error of $\frac{1}{k^2}$ at most. 
\item There are remaining components not of the first two types, but the number of them is $\val(v_i)-1$, whch is independent of $k$. Moreover, they are arbitrarily close to $R^{-1}(v_i)$ when we increase $k$.
In particular, the oscillation of $H$ over these components converges to $0$ when $k$ goes to infinity.
\end{enumerate}

For each vertex $v_i \in G$,  let $\chi_i:=2-\val(v_i)$.
We are now going to use the monotone links $\ul{L}_k$ from Construction \ref{c:linkExist} to prove the following.

\begin{prop}\label{p:ruelle}
Let $H:(\S^2,\omega) \to \mathbb{R}$ be an autonomous Hamiltonian with finitely many critical  values. Then
\begin{align}
\lim_{k \to \infty} \left( k\mu_{k}(H)-(k+1) \int_{\S^2} H\right) =  -\frac{1}{2}\sum_{i=1}^t \chi_i H_G(v_i) \label{eq:formulaS2}
\end{align}
\end{prop}

\begin{rem}
When $H$ is Morse, \eqref{eq:formulaS2} reduces to
\begin{align*}
\lim_{k \to \infty} \left( k\mu_{k}(H)-(k+1) \int_{\S^2} H \right) = -\frac{1}{2}\sum_{j=1}^s (-1)^{\mathrm{ind}(p_j)} H(p_j)
\end{align*}
where $\{p_j\}_{j=1}^s$ is the set of critical points of $H$. Compare to Equation \eqref{eqn:Ruelle_Morse_case}.
\end{rem}

\begin{proof}[Proof of Proposition \ref{p:ruelle}]
For every $k \in \mathbb{N}$ such that $\frac{3}{k+1}<\mu(e)$, we apply Construction \ref{c:linkExist} to obtain a monotone link $\ul{L}_k$ as above.
We continue to use the notations $V_i$, $J_i$, $T_1$, $T_{2,i}$, $T_{3,i}$, etc, but they should be understood that they depend on $k$.

The Lagrangian control property \ref{eq:strongerLagcontrol}
applied to the link $\ul{L}_k$ yields
\begin{align}
k\mu_k(H) \in \sum_{j=1}^k H(L_{k,j}),  \label{eq:LagControlmu}
\end{align}
where $H(L_{k,j}):=\{H(y)| y \in L_{k,j}\}$.

We want to show the following three identities corresponding to the three types of circles:
\begin{align}
&\lim_{k \to \infty} \left(\left( \sum_{j \in T_{1}} H(L_{k,j}) \right)-(k+1) \int_{\S^2 \setminus \cup_{i=1}^t R^{-1}(V_i)} H \right) =\frac{1}{2} \sum_{i=1}^t  \val(v_i) H_G(v_i),  \label{eq:est}\\
&\lim_{k \to \infty} \left(\left( \sum_{j \in T_{2,i}} H(L_{k,j}) \right)- (k+1) \int_{R^{-1}(V_i)} H \right)=-\val(v_i)H_G(v_i),  \text{ for all }i, \label{eq:collapsing1}\\
&\lim_{k \to \infty} \sum_{j \in T_{3,i}} H(L_{k,j})=(\val(v_i)-1) H_G(v_i), \text{ for all }i. \label{eq:collapsing2}
\end{align}
Note that these are limits of sets, since not all the $L_{k,j}$ are contained in level sets of $H$; however, the $J_{i}$ shrink with $k$ and hence the diameters of the sets $H(L_{k,j})$ also tend to zero as $k$ increases to infinity.  Once these equalities are proved, by summing them up
we will get 
\begin{align}
\lim_{k \to \infty} \left( \sum_{j=1}^k H(L_{k,j}) -(k+1) \int_{\S^2} H\right) = \frac{1}{2} \sum_{i=1}^t  (\val(v_i)-2) H_G(v_i) .\label{eq:cestimate}
\end{align}
The result then follows from \eqref{eq:LagControlmu} and the observation that the RHS of \eqref{eq:cestimate} is precisely $-\frac{1}{2}\sum_{i=1}^t \chi_i H_G(v_i)$.

We should first explain the basic reasoning behind the identities \eqref{eq:est}, \eqref{eq:collapsing1} and \eqref{eq:collapsing2}.
Since type $T_1$ circles are $S^1$ fibres over the open edges, \eqref{eq:est} essentially follows from a classical disintegration statement which computes the difference between the integral of a function over an interval and its Riemann sum.
On the other hand, combining \eqref{eq:collapsing1} and \eqref{eq:collapsing2} gives that the total contribution of type $T_{2,i}$ and $T_{3,i}$ circles is roughly $(k+1) \int_{R^{-1}(V_i)} H$ (offset by the constant $H_G(v_i)$ which is independent of $k$).
This is essentially because the oscillation of $H$ over type $T_{2,i}$ circles is bounded above by $\frac{1}{k^2}$
and the oscillation of $H$ over type $T_{3,i}$ circles converges to $0$.
The computations below are to carefully obtain the correct coefficients in front of $H_G(v_i)$ in these identities.

Identity \eqref{eq:est} can be proved `edge by edge'.
More precisely, let $E$ be the set of edges of $G$.
Each connected component of $\S^2 \setminus \bigcup_{j \in T_{1}} L_{k,j}$  other than $V_{k,1},\dots,V_{k,t}$ is topologically an annulus and is canonically labeled by an edge $e \in E$.
Let $A^e_{k,1},\dots,A^e_{k,h_{e,k}}$ be the connected components
that are labeled by $e$ and denote their closures by $\overline{A}^e_{k,i}$ for $i=1,\dots,h_{e,k}$.
By possibly relabeling, we can assume that $\overline{A}^e_{k,i} \cap \overline{A}^e_{k,j} \neq \emptyset$ if and only if $j \in \{i-1,i,i+1\}$.
Let $\partial \overline{A}^e_{k,i}=L^e_{k,i-1} \cup L^e_{k,i}$ for $i=1,\dots,h_{e,k}$.
Note that, every component of $\bigcup_{j \in T_{1}} L_{k,j}$ is of the form $L^e_{k,i}$ for precisely one $e \in E$ and $i \in \{0,\dots,h_{e,k}\}$.
By identifying $A^e_{k,i}$ with $([\frac{i-1}{k+1},\frac{i}{k+1}] \times \mathbb{R}/\mathbb{Z}, dz \wedge dy)$ using an $S^1$-equivariant area preserving diffeomorphism, in such a way that $L^e_{k,i}$ is identified with $\{i\} \times \mathbb{R}/\mathbb{Z}$, we have
\begin{align*}
&\sum_{i=1}^{h_{e,k}} \left( H(L^e_{k,i})-(k+1) \int_{A^e_{k,i}} H \right)\\
=&\, (k+1)\sum_{i=1}^{h_{e,k}} \int_{\frac{i-1}{k+1}}^{\frac{i}{k+1}} -H'(\frac{i}{k+1})(z-\frac{i}{k+1}) + O\left(\frac{1}{(k+1)^2}\right)dz\\
=&\, \left(\frac{1}{2(k+1)} \sum_{i=1}^{h_{e,k}} H'(\frac{i}{k+1}) \right)+ h_{e,k}\, O\left(\frac{1}{(k+1)^2}\right)
\end{align*}
where Taylor's theorem was used in passing from the first to the second line above.
Note that $h_{e,k} < (k+1)$.
By passing $k$ to $\infty$ and applying the fundamental theorem of calculus, we get
\begin{align*}
\lim_{k \to \infty} \sum_{i=1}^{h_{e,k}} \left(H(L^e_{k,i})-(k+1) \int_{A^e_{k,i}} H\right) 
=&\lim_{k \to \infty} \frac{1}{2} (H(L^e_{k,h_{e,k}})-H(L^e_{k,0}))\\
=&\frac{1}{2} (H_G(\partial^1 e)-H_G( \partial^0 e))
\end{align*}
where $\partial^i e$ are the corresponding vertices adjacent to $e$.
By adding back the term $\lim_k H(L^e_{k,0})$, we have
\begin{align*}
\lim_{k \to \infty} \left(\sum_{i=0}^{h_{e,k}} H(L^e_{k,i}) -\sum_{i=1}^{h_{e,k}}(k+1) \int_{A^e_{k,i}} H\right)
=&\frac{1}{2} (H_G(\partial^1 e)+H_G( \partial^0 e))
\end{align*}
This completes the calculation over a single edge $e$. By summing over all $e \in E$, we get  \eqref{eq:est}.

Equation \eqref{eq:collapsing2} follows from the fact that $|T_{3,i}|=\val(v_i)-1$ and the fact that $L_{k,j}$, for $j \in T_{3,i}$,  is approaching to $R^{-1}(v_i)$ as $k$ goes to infinity.
Therefore, it suffices to verify Equation \eqref{eq:collapsing1}.

Recall that the oscillation of $H$ over type $T_{2,i}$ circles is bounded above by $\frac{1}{k^2}$.
It implies that 
\begin{align*}
\sum_{j \in T_{2,i}} H(L_{k,j}) \subset \left[S_{k,i}\left(H_G(v_i)-\frac{1}{k^2}\right), S_{k,i}\left(H_G(v_i)+\frac{1}{k^2}\right)\right]
\end{align*}
On the other hand, recall that $\mu(v_i)=m_i \ge \frac{S_{k,i}}{k+1}$.
Therefore, we can find an open disc $D_{k,i} \subset R^{-1}(J_{i})$ such that $\omega(D_{k,i})=\frac{S_{k,i}}{k+1}$.
It implies that 
\begin{align}
\lim_{k \to \infty} \left(\left( \sum_{j \in T_{2,i}} H(L_{k,j}) \right)- (k+1) \int_{D_{k,i}} H \right)=\lim_{k \to \infty} [-\frac{1}{k^2}S_{k,i},\frac{1}{k^2}S_{k,i}]=0  \text{ for all }i  \label{eq:break1}
\end{align}
Since $\omega(R^{-1}(V_i)\setminus D_{k,i})=\omega(R^{-1}(V_i))-\omega(D_{k,i})=\frac{S_{k,i}+\val(v_i)}{k+1}-\frac{S_{k,i}}{k+1}=\frac{\val(v_i)}{k+1}$, we also have 
\begin{align}
\lim_{k \to \infty}  (k+1) \int_{R^{-1}(V_i)\setminus D_{k,i}} H=\lim_{k \to \infty} (k+1)\omega(R^{-1}(V_i)\setminus D_{k,i})H_G(v_i)=\val(v_i)H_G(v_i)  \text{ for all }i  \label{eq:break2}
\end{align}
Equation \eqref{eq:collapsing1} now follows from Equations \eqref{eq:break1} and \eqref{eq:break2}.
\end{proof}

\begin{proof}[Proof of Theorem \ref{t:ruelle}]
We embed $\bD^2$ into the northern hemisphere of $\S^2$.
By Proposition \ref{p:ruelle} and \eqref{eq:RuelleFormula}, it suffices to show that $\sum_{i=1}^t \chi_i H_G(v_i)$ coincides with $\int_{\bR} n_H(\xi) d\xi$.
We can reinterpret $\int_{\bR} n_H(\xi) d\xi$ using $G$ as follows.
Let $v_1$ be the vertex of $G$ given by $R(\partial \bD^2)$ (it is also the $R$-image of the entire southern hemisphere).
Let $e$ be an edge of $G$.
Let the two vertices adjacent to $e$ be $\partial^+ e$ and $\partial^- e$. Since $G$ is a tree, there is no ambiguity to require that $\partial^+ e$ is further away from $v_1$
than $\partial^- e$ (we allow that $\partial^- e=v_1$).
If $H_G(\partial^+ e) > H_G(\partial^- e)$, we define $n_e=1$.
If $H_G(\partial^+ e) < H_G(\partial^- e)$, we define $n_e=-1$.
It is clear from the definition of $n_H$ that
\begin{align*}
\int_{\bR} n_H(\xi) d\xi=\sum_{e} \int_{H_G(e)} n_e= \sum_e (H_G(\partial^+ e)-H_G(\partial^- e)) \label{eq:anotherSum}
\end{align*}
where the sum is over all edges of $G$.
For any vertex $v$ of $G$ other than $v_1$, there is a unique edge $e_v$ of $G$ such that $\partial^+ e_v=v$ because $G$ is a tree.
Therefore, we have 
\begin{align*}
\sum_e (H_G(\partial^+ e)-H_G(\partial^- e))
=&-\val(v_1)H_G(v_1)+\sum_{i=2}^t \left( H_G(v_i)-(\val(v_i)-1)H_G(v_i) \right)\\
=&\sum_{i=1}^t \chi_i H_G(v_i)
\end{align*}
where the last equality uses that $H_G(v_1)=0$ and $\chi_i=2-\val(v_i)$.
It completes the proof.
\end{proof}

\begin{rem}
For higher genus surfaces, one can use a similar method to estimate $c_{\ul{L}}(H)$ for appropriate monotone links $\ul{L}$ most of whose components are level sets of $H$.
However, the homogenized spectral invariant $\mu_{\ul{L}}$ will depend on the particular link $\ul{L}$, and not only on the number of components of $\ul{L}$.
Therefore, for any fixed sequence of monotone links $\{\ul{L}_k\}_{k \in \mathbb{N}}$, this method should not give a robust estimate of the subleading asymptote of $\mu_{\ul{L}_k} $  for all autonomous Hamiltonians $H$ simultaneously.
\end{rem}



\section{Non-simplicity for kernel of Calabi and for Hameo}\label{s:non-simple}

In this section, we prove Theorem \ref{thm:notsimple} whose statement we recall here.  

\begin{thm*}[Theorem \ref{thm:notsimple}]
The following groups are not perfect:

\begin{enumerate}
\item The kernel of Calabi on $\mathrm{Hameo}(\mathbb{D}^2,\omega)$.
\item The group $\mathrm{Hameo}(\mathbb{S}^2,\omega)$.
\end{enumerate}

Both admit surjective group homomorphisms to $\mathbb{R}$.

\end{thm*}

The goal of this section is to explain the proof.  The broad strategy of the proof is as follows. Let $G$ denote either $\mathrm{Homeo}_{\mathrm{c}}(\mathbb{D}^2,\omega)$ or $\mathrm{Homeo}(\mathbb{S}^2,\omega)$ and let $H$ be any of the two groups in the statement of Theorem \ref{thm:notsimple}.  We begin by defining certain normal subgroups of $H$, denoted by $N(\mathbb{D}^2)$ and $N(\mathbb{S}^2)$ respectively, which will turn out to be proper.   Now, these groups are also normal in $G$ and since any normal subgroup of $G$ contains $[G,G]$ (see \cite[Prop.\ 2.2]{CGHS20}\footnote{\label{footnote:[G,G]}Proposition 2.2 in \cite{CGHS20} is only stated on the disc, but holds on any compact surface by the same argument.}), we conclude that $H$ is not perfect.  The calculation of a quotient isomorphic to 
$\mathbb{R}$ proceeds readily from this, as we will explain.

In the rest of this section we define the normal subgroups $N(\mathbb{D}^2), N(\mathbb{S}^2)$  and prove their properness.

\subsection*{Proper normal subgroups from subleading asymptotics}
To define our normal subgroups, we will use the subleading asymptotics of the quasimorphims  arising from link spectral invariants which were introduced in Section \ref{sec:prelim-quasis}.


First, consider the case of the disc. Denote by $\mathrm{ker(Cal)}$ the kernel of the Calabi homomorphism $\mathrm{Cal} : \mathrm{Hameo}(\mathbb D^2, \omega) \rightarrow \mathbb R.$  Recall the quasimorphism $f_k :\mathrm{Homeo}(\mathbb S^2, \omega) \rightarrow \R$; its restriction to $\mathrm{Homeo}_{\mathrm{c}}(\mathbb D^2, \omega)$ has defect  bounded by $\frac{2}{k}$, see Lemma \ref{l:quasi2k}.  Our normal subgroup will consist of those elements (of the kernel of Cal) for which the $f_k$ have bounded subleading asymptotics. More precisely, define 
$$N(\mathbb{D}^2) := \{ \psi \in \mathrm{ker(Cal)}: \text{the sequence }  \vert kf_k(\psi) \vert \text{ is bounded} \}.$$

\begin{prop}\label{prop:N-disc}
$N(\mathbb{D}^2)$ is a normal subgroup of $\mathrm{ker(Cal)}$ which contains all of its smooth elements. Moreover, it is also normal in $\mathrm{Homeo}_{\mathrm{c}}(\mathbb D^2, \omega)$.
\end{prop}
\begin{proof}
The argument here is very similar to that of the proof of Theorem \ref{t:muO(1)} and so we will not provide all the details.  $N(\mathbb{D}^2)$ is a subgroup because the $f_k$ have defect $\frac{2}{k}$  and it is normal in $\mathrm{Homeo}_{\mathrm{c}}(\mathbb D^2, \omega)$ because $\mathrm{ker(Cal)}$ is a normal subgroup and because the $f_k$, being homogeneous quasimorphisms, are invariant under conjugation.

The fact that $N(\mathbb{D}^2)$ contains all of the smooth elements in the kernel of Calabi is a consequence of Corollary \ref{corol:O(1)}; this is because for such $\psi$, we have $\mu_1(\psi) = \mathrm{Cal}(\psi) = 0$.   
\end{proof}

In the case of the sphere, our normal subgroup is defined similarly, however we cannot use the quasimorphisms $f_k : \mathrm{Homeo}(\mathbb S^2, \omega) \rightarrow \mathbb R$ because although the restriction $f_k : \mathrm{Homeo}_{\mathrm{c}}(\mathbb D^2, \omega) \rightarrow \mathbb R$ has defect $\frac{2}{k}$, the $f_k$ have defect $\frac{2}{k} + 2$; see Remark \ref{rem:difference-of-qm}.
We remedy this problem by working instead with the sequence of quasimorphisms
\begin{equation}\label{eqn:gk}
g_k := \mu_{2^k - 1} - \mu_{2^{k-1} - 1},
\end{equation}   
for $k\geq 2$  on $\Homeo(\S^2,\omega)$.  Then, the defect of $g_k$ is bounded by $\frac{2}{2^k - 1}+\frac{2}{2^{k-1} - 1}$, ; see Remark \ref{rem:difference-of-qm},  which in particular converges to $0$ as $k$ goes to infinity.  (While many other differences $\mu_{\alpha(k)} - \mu_{\alpha(k-1)}$ would also have defect limiting to 0, the choice $\alpha(k) = 2^k - 1$ will be particularly convenient for our calculations in Section 4.3.)

Define 
$$N(\mathbb{S}^2) := \{ \psi \in \mathrm{Hameo}(\mathbb{S}^2, \omega) :  \text{the sequence }  \vert (2^k - 1) g_k(\psi) \vert \text{ is bounded}  \}.$$

\begin{prop}\label{prop:N-sphere}
$N(\mathbb{S}^2)$ is a normal subgroup of $\mathrm{Hameo}(\mathbb S^2, \omega)$ which contains all of its smooth elements. Moreover, it is also normal in $\mathrm{Homeo}(\mathbb S^2, \omega)$.
\end{prop}
  As in the case of $N(\mathbb{D}^2)$, the proof of the above is similar to that of Theorem \ref{t:muO(1)} and so we will omit it. 
  
  \medskip

To prove properness of these normal subgroups, we will exhibit examples of hameomorphisms with unbounded subleading asymptotics.

\subsection{A quickly twisting hameomorphism}

The first part of the proof is to find a useful element that is in Hameo.  As in our previous work, \cite{CGHS20, CGHS21, CGHMSS21}, the desired map will be a twist map.  However, in our previous work, we studied ``infinite twists" that were twisting so quickly that they were not in Hameo.  Here, we find a map that is twisting slowly enough to define an element of Hameo, but quickly enough to have interesting, i.e.\ unbounded, subleading asymptotics.  The construction of this map will be the topic of this section.

\medskip

 Let $T: \mathbb S^2 \to \mathbb S^2$ be defined as follows.  
 
We view $\mathbb S^2$ as the standard unit sphere $\{(x,y,z) \in \mathbb R^3: x^2 + y^2 + z^2 = 1 \}$ in $\R^3$ and equip it with the symplectic form $\omega =  \frac{1}{4\pi} d\theta \wedge dz $ where $(\theta, z)$ are the standard cylindrical coordinates on $\mathbb{R}^3$. Denote by $p_-$ the  point on $\mathbb S^2$ whose $z$--coordinate is $-1$.  We pick a function $H: \mathbb{S}^2 \setminus \{p_-\} \rightarrow\R$ which is of the form  

 $$H(\theta, z) =  h(z),$$
where $h: (-1,1]  \rightarrow \mathbb R$ is a smooth function which vanishes for $z\geq - \frac{1}{2}$ and, for $z \leq - \frac{3}{4}$, satisfies the identity 

\begin{equation}
 \label{eqn:hamT} 
 h(z) = \sqrt{\frac{2}{1+z}}.
 \end{equation}

 The function $H$ induces a well-defined flow $\phi^t_H$ on $\mathbb S^2$ which fixes the point $p_-$ and its action on $(\theta, z)$, with $z>-1$, is given by the following equation
  \begin{equation*}
 \phi^t_H(\theta, z) = (\theta + 4\pi h'(z)t, z).
 \end{equation*}
  We define $$T := \phi^1_H.$$
 
 Note that $T$ is supported in the disc $\mathbb{D}^2 := \{(\theta, z) : -1 \leq z \leq 0 \} \subset \mathbb{S}^2$ and so we can view it as an element of either of $\mathrm{Homeo}_{\mathrm{c}}(\mathbb D^2, \omega)$ or $\mathrm{Homeo}(\mathbb S^2, \omega)$.

\begin{prop}\label{prop:T-hameo}
$T \in \mathrm{Hameo}(\mathbb{D}^2,\omega)$.  Moreover, $$\mathrm{Cal}(T) = \frac{1}{2} \int_{-1}^1 h(z) dz < \infty.$$
\end{prop}

Note that the above proposition implies that $T \in \mathrm{Hameo}(\mathbb S^2,\omega)$ as well.  

\begin{rem} \label{rem:fast-vs-toofast}
As mentioned earlier, the homeomorphism $T$ twists slowly enough to be in $\Hameo$, and so its Calabi invariant is well-defined, yet it twists fast enough  not to be contained in $N(\mathbb D^2)$; the heuristic reasoning behind $T\notin N(\mathbb D^2)$ is that, since $H(p_-) = \infty$, $T$ has ``infinite Ruelle invariant." 

 In comparison, if we were to modify the function $h$ in Equation \eqref{eqn:hamT} to $$h(z) =\frac{2}{1+z}$$
we would obtain an ``infinite twist" homeomorphism that spins too fast to be contained in $\Hameo$; here  the heuristic reasoning is that the condition $\int_{-1}^1 h(z) dz = \infty$ forces the homeomorphism to have ``infinite Calabi invariant."  Indeed, this can be proven rigorously via the argument given in \cite{CGHS20} (see also the proof of \cite[Theorem 1.3]{CGHMSS21}).
\end{rem}

\begin{proof}[Proof of Proposition \ref{prop:T-hameo}]
By definition of $\mathrm{Hameo}$, to prove that  $T \in \mathrm{Hameo}(\mathbb{D}^2,\omega)$, we must find smooth Hamiltonians $K_n$ supported in a compact subset of the interior of $\mathbb{D}^2 = \{(\theta, z) : -1 \leq z \leq 0 \}$ such that
\begin{enumerate}[(A)]
\item $\phi^1_{K_n} \xrightarrow{C^0} T$,
\item $\phi^t_{K_n}$ is Cauchy for the $C^0$-distance, uniformly in $t\in[0,1]$,
\item the sequence $K_n$ is Cauchy for Hofer's norm $\|\cdot\|$.
\end{enumerate}

We start by picking Hamiltonians $H_n$ as follows.  Let $D_n := \{ (\theta, z): -1  \leq z \leq -1 + \varepsilon_n \} \subset \mathbb{D}^2$ be the disc of radius $\varepsilon_n := \frac{1}{2^{2n}}$ (in $z$ coordinate)  centered at $p_-$; note that 
\begin{equation}\label{eqn:area_Dn}
\mathrm{Area}(D_n) = \varepsilon_n \mathrm{Area}(\mathbb D^2) =  \frac{\varepsilon_n}{2} .
\end{equation} 

Now, pick the Hamiltonian $H_n$ so that the following hold:
\begin{enumerate}[(i)]
\item $H_n$ depends only on the $z$ variable,
\item  $H_n = H$ outside of $D_{n}$ and $H_n \approx  \sqrt{\frac{2}{\varepsilon_n}}  $ in the interior of $D_n$, 
\item $\Vert H_{n+1} - H_n \Vert \leq \sqrt{\frac{2}{\varepsilon_n}}$.
\end{enumerate}
To see why $H_n$ can be picked to satisfy the above, note that $H( -1 + \varepsilon_n) =  \sqrt{\frac{2}{\varepsilon_n}}$ and so to obtain $H_n$ it suffices to smoothly flatten $H$ on the interior of $D_n$.

Note that $\phi^1_{H_n} \circ T^{-1} = \mathrm{Id}$ outside of $D_n$ and hence $\phi^1_{H_n} \xrightarrow{C^0} T$.  We will find Hamiltonians $K_n$ such that $\phi^1_{K_n}=\phi_{H_n}^1$, the sequence $K_n$ is Cauchy for Hofer's norm $\|\cdot\|$ and $\phi_{K_n}^t$ is Cauchy for the $C^0$-distance, uniformly in $t$.  Note that once this is proven Theorem \ref{th:extension-calabi} yields 
$$\mathrm{Cal}(T) = \lim_n \mathrm{Cal}(\phi^1_{K_n}) =   \lim_n \mathrm{Cal}(\phi^1_{H_n})= \lim_n \int_{\mathbb D^2} H_n \, \omega = \int_{\mathbb D^2} H \; \omega =   \frac{1}{2} \int_{-1}^1 h(z) dz.$$

We need the following lemma whose proof relies on ideas going back to Sikorav \cite{Sikorav-Pisa}.

\begin{lem}\label{lem:sikorav-trick} Let $\Delta$ be a Euclidean 2-disc equipped with an area form $\omega$ of total area $A$. Suppose $D\subset \Delta$ is diffeomorphic to $\D^2$ and that $\mathrm{Area}(D) < \frac{A}{N}$ for some integer $N>0$.  Let $F$ be a smooth Hamiltonian supported in the interior of $D$. Then,  we have
\begin{align*}
d_H(\phi^1_F, \mathrm{Id}) \leq \frac{\Vert F \Vert}{N} + 2A.
\end{align*} 
 where $d_H$ denotes the Hofer distance on $\mathrm{Ham}_{\mathrm{c}}(\Delta, \omega)$  and $\|F\|=\int_0^1(\max_{\Delta}F_t-\min_{\Delta}F_t)dt$  is the Hofer norm of $F$.
\end{lem}

Before proving this lemma, we will use it to construct the sequence of Hamiltonians $K_n$.

For each $n$, the Hamiltonian $H_{n+1} - H_n$ is supported in the disc $D_n$, by item (ii) above, and   $\|H_{n+1} - H_n \| \leq \sqrt{\frac{2}{\varepsilon_n}}$, by (iii).
Let $\Delta_n\subset \D^2$ be the disc centered at $p_-$ and of area $A_n:=2^{-n/2}$. By Equation \eqref{eqn:area_Dn}, we have $\mathrm{Area}(D_n) < \frac{A_n}N$ for $N:=2^{\lfloor 3n/2\rfloor}$. Hence, applying Lemma \ref{lem:sikorav-trick}, we obtain Hamiltonians $G_n$ supported in $\Delta_n$ which satisfy
\begin{itemize}
\item $\phi_{G_n}^1=\phi_{H_{n+1}-H_n}^1=\phi_{H_n}^{-1}\phi_{H_{n+1}}^1$,
 \item $\|G_n\| \leq \frac{\|H_{n+1}-H_n\|}{N} + 2 A_n \leq \frac{ \sqrt{ \frac{2}{\varepsilon_n}}}{N} + 2 A_n $.
\end{itemize}

Note that $\frac{ \sqrt{ \frac{2}{\varepsilon_n}}}{N} + 2 A_n = \frac{2^n \sqrt{2}}{2^{\lfloor 3n/2\rfloor}} + \frac{2}{2^{\frac{n}{2}}}$ and thus the series $\sum_{i=1}^\infty \|G_i\|$ is summable.
Since $G_n$ is supported in $A_n$, the $C^0$-distance $d_{C^0}(\phi_{G_n}^t,\id)$ is bounded by the diameter of $A_n$, which is $O(2^{-n/4})$. It follows that the series $\sum_{i=1}^\infty d_{C^0}(\phi_{G_i}^t,\id)$ is summable as well (uniformly in $t$).

Now let us define $K_1:=H_1$ and then recursively $K_{n+1}:=K_n\sharp G_n$ for $n\geq 1$. Then, \[\phi_{K_n}^1=\phi_{H_1}^1\phi_{G_{1}}^1\cdots\phi_{G_{n-1}}^1=\phi_{H_n}^1.\] Moreover, since $\sum_{i=1}^\infty \|G_i\|=\sum_{i=1}^\infty \|K_{i+1}-K_i\|$ is summable, the sequence $K_n$ is Cauchy with respect to the Hofer norm. Similarly, since $\sum_{i=1}^\infty d_{C^0}(\phi_{G_i}^t,\id)$ is summable, $\phi_{K_n}^t$ converges for the $C^0$ topology.

 This completes the proof of Proposition \ref{prop:T-hameo} modulo the proof of the lemma which we provide below. \end{proof}
\begin{proof}[Proof of Lemma \ref{lem:sikorav-trick}] We will present the proof of the lemma under the simplifying assumption that the Hamiltonian $F$ is time independent and leave the more general case, which is very similar, to the reader.  Note that we have only applied Lemma \ref{lem:sikorav-trick} to time-independent Hamiltonians.  

Pick pairwise disjoint discs $D_1, \ldots, D_N \subset \Delta$ such that each of these discs has the same area as $D$.  There exist Hamiltonian diffeomorphisms $\psi_1, \ldots, \psi_N \in \Ham_{\mathrm{c}}(\Delta, \omega)$ such that
\begin{itemize}
\item $\psi_i(D) = D_i$ for each $i = 1, \ldots, N$,
\item $d_H(\psi_i, \mathrm{Id}) \leq \frac{A}{N}$. 
\end{itemize}

Consider the time-independent Hamiltonian $$H:= \frac{1}{N} \sum_{i=1}^N F \circ \psi_i^{-1}.$$  It is supported in the union of the discs $D_i$ and $\Vert H \Vert \leq \frac{\Vert F \Vert}{N}$.  Therefore,

$$d_H(\phi^1_F, \mathrm{Id}) \leq d_H(\phi^1_F, \phi^1_H) + d_H(\phi^1_H, \mathrm{Id} ) \leq  d_H(\phi^1_F, \phi^1_H) +  \frac{\Vert F \Vert}{N}.$$

Hence, to prove the lemma, it is sufficient to show that $d_H(\phi^1_H, \phi^1_F) \leq 2A$.  To do so, first observe that $ \phi^1_H = \prod_{i=1}^N \, \psi_i \phi^{1/N}_{F} \psi_i^{-1} $ and $\phi^1_F =\prod_{i=1}^N  \, \phi^{1/N}_{F}$. Hence,
\begin{align*}
d_H(\phi^1_H, \phi^1_F) &= d_H\left(\prod_{i=1}^N \, \psi_i \phi^{\frac{1}{N}}_{F} \psi_i^{-1},  \prod_{i=1}^N  \, \phi^{\frac{1}{N}}_{F}\right) \\
&\leq \sum_{i=1}^N d_H\left(\phi^{\frac{N-i}{N}}_F \prod_{j=1}^i \psi_j \phi^{\frac{1}{N}}_{F} \psi_j^{-1} , \phi^{\frac{N-i+1}{N}}_F \prod_{j=1}^{i-1} \psi_j \phi^{\frac{1}{N}}_{F} \psi_j^{-1}\right)\\
&= \sum_{i=1}^N d_H(\psi_i \phi^{\frac{1}{N}}_{F} \psi_i^{-1} , \phi^{\frac{1}{N}}_{F})\\
&\leq \sum_{i=1}^N d_H(\psi_i \phi^{\frac{1}{N}}_{F} \psi_i^{-1}, \phi^{\frac{1}{N}}_{F} \psi_i^{-1} ) +d_H(\phi^{\frac{1}{N}}_{F} \psi_i^{-1}, \phi^{\frac{1}{N}}_{F}) \\
&= \sum_{i=1}^N d_H(\psi_i, \mathrm{Id }) + d_H(\psi_i^{-1}, \mathrm{Id })\\
& = \sum_{i=1}^N 2 \, d_H(\psi_i, \mathrm{Id })  \leq  2A \end{align*} 
The inequalities on the second and fourth lines follow from the triangle inequality where when $i=1$, $\phi^{\frac{N-i+1}{N}}_F \prod_{j=1}^{i-1} \psi_j \phi^{\frac{1}{N}}_{F} \psi_j^{-1}$ should be understood as $\phi^{\frac{N}{N}}_F=\phi^1_F$.
The equalities on the third and fifth lines follow from the bi-invariance of Hofer's metric and the inequality on the final line follows from the fact that we picked $\psi_i$ such that $d_H(\psi_i, \mathrm{Id}) \leq \frac{A}{N}$. 
\end{proof}

\subsection{The case of the disc}\label{sec:disc}
We now use the map $T$ to prove that 
the group $N(\D^2)$ from Proposition \ref{prop:N-disc} is proper in the kernel of Calabi on $\mathrm{Hameo}(\mathbb{D}^2,\omega)$. 

\begin{proof}[Proof of properness of $N(\D^2)$.]
 We will do this in two steps.

\medskip

{\bf Step 1.}  We claim that the sequence $\displaystyle  | k  ( f_k(T) -\mathrm{Cal}(T) )  |$ is unbounded, where $T \in \mathrm{Hameo}(\mathbb D^2, \omega)$ is as in Proposition \ref{prop:T-hameo}.

Recall the (non-smooth) function $H$ from \eqref{eqn:hamT} which we used in the definition of $T$. Let $H_n(z)$ be a sequence of smoothings of $H$,  depending only on $z$, that agree with $H$ except for $-1 \leq z \leq -1 + \frac{1}{2^{2n}} $.  One could, for example, take $H_n$ to be as in the proof of Proposition \ref{prop:T-hameo}.

Note that $\phi^1_{H_n} \xrightarrow{C^0} T$ and so, by the $C^0$ continuity property of the $f_k$, we have $$f_k(T)  = \lim_{n \to \infty} f_k (\phi^1_{H_n}).$$
Now, since $H$ and the $H_n$ depend only on $z$, we can compute $f_k(\phi^1_{H_n})$ using the Lagrangian Control property; see Section \ref{sec:prelim-quasis}.  We have $$f_k(\phi^1_{H_n}) = \frac{1}{k} \sum_{i=1}^k H_n(-1 + 2\frac{i}{k+1}).$$
Since $H_n = H$, except for $-1 \leq z \leq -1 + \frac{1}{2^{2n}}$, for $n$ large enough we have 
$f_k(\phi^1_{H_n}) = \frac{1}{k} \sum_{i=1}^k H(-1 + 2\frac{i}{k+1})$ and hence 

$$f_k(T) = \frac{1}{k} \sum_{i=1}^k H(-1 + 2\frac{i}{k+1}) = \frac{1}{k} \sum_{i=1}^k h(-1 + 2\frac{i}{k+1}).$$

Recall that $\mathrm{Cal}(T) = \frac{1}{2} \int_{-1}^1 h(z) \, dz < \infty$.
 Hence, proving that $k f_k(T) - k \mathrm{Cal}(T)$ is unbounded, is equivalent to proving that the sequence whose kth term is given by  
\begin{equation}\label{eqn:prove_unbdd}
2(k f_k(T) - (k+1) \mathrm{Cal}(T))=2\sum_{i=1}^k h(-1 + 2\frac{i}{k+1}) - (k+1) \int_{-1}^1 h(z) \, dz
\end{equation} 
 is unbounded; we will prove this below.

Write $a_i = -1 + 2\frac{i}{k+1}$, for $ i = 0, \ldots, k+1$.  Observe that \eqref{eqn:prove_unbdd} can be rewritten as 
\begin{align*}
(k+1) \sum_{i=1}^k   \left( \int_{a_{i-1}}^{a_{i}} h( a_i) - h(z)  dz \right) \; - (k+1)\int_{a_k}^{a_{k+1}} h(z) dz.
\end{align*}
The term  $\int_{a_k}^{a_{k+1}} h(z) dz$ is zero since $h$ is supported in $-1 \leq z \leq - \frac{1}{2}$. So we must prove unboundedness of the sum.

Since $h$ is a convex function, we have $h( a_i) - h(z) \leq  h'(a_i) (a_i-z)$. Thus

\begin{align*}
(k+1) \sum_{i=1}^k   \left( \int_{a_{i-1}}^{a_{i}} h( a_i) - h(z)  dz \right)  &\leq
(k+1) \sum_{i=1}^k  \left( h'(a_i)  \int_{a_{i-1}}^{a_{i}} (a_i-z) dz \right) \\
                                                                           = (k+1) \sum_{i=1}^k \left( h'(a_i) \frac{2}{(k+1)^2} \right) 
                                                                           &= \sum_{i=1}^k \left( h'(a_i) \frac{2}{k+1} \right).
\end{align*}

Now, since $h'$ is non-decreasing, we have
\begin{align*}
  \sum_{i=1}^k \frac{2}{k+1} h'(a_i)\leq  \sum_{i=1}^k \int_{a_i}^{a_{i+1}}h'(z)dz
  =\int_{a_1}^1 h'(z)dz = -h(a_1) \stackrel{k\to \infty}{\longrightarrow} -\infty.
\end{align*}
This shows that 
$(k+1) \sum_{i=1}^k   \left( \int_{a_{i-1}}^{a_{i}} h( a_i) - h(z)  dz \right) $
is unbounded and concludes the proof of unboundedness of the sequence in \eqref{eqn:prove_unbdd}.

\medskip
 
 {\bf Step 2.} Let $\Theta \in \mathrm{Diff}_{\mathrm{c}}(\mathbb{D}^2,\omega)$ be such that 
 $\mathrm{Cal}(\Theta) =\mathrm{Cal}(T)$ 
 and define
\[ \psi' =  T \circ \Theta^{-1}.\]
Then, $\mathrm{Cal}(\psi') = 0$.  We will show that the sequence $\vert k f_k(\psi') ) \vert$ is unbounded which implies that $\psi' \notin N(\mathbb D^2)$ and hence establishes properness of $N(\mathbb D^2)$.

By Lemma~\ref{lem:defect} 
\[ |k f_k(\psi') - k f_k(T) - k f_k(\Theta^{-1})| \le 2 .\]
Now, we claim that the sequence $ \{ k f_k(T) + k f_k(\Theta^{-1}) \}_k$ is unbounded which, in combination with the above inequality, implies that the sequence $\vert k f_k(\psi') ) \vert$ is unbounded.  

The fact that $ \{ k f_k(T) + k f_k(\Theta^{-1}) \}_k$ is unbounded  is an immediate consequence of Theorem \ref{t:muO(1)}:  the sequence $\{  k \cdot (f_k(\Theta^{-1}  ) -\mathrm{Cal}(\Theta^{-1}) )\}_k$ is bounded, by the theorem, and the sequence $\{ k \cdot (f_k(T) -\mathrm{Cal}(T) ) \}_k$ is unbounded, by Step 1. Hence, the sum of these two sequences, which is exactly $\{ k f_k(T) + k f_k(\Theta^{-1})\}_k$, is unbounded.   This completes the proof of the fact that $\psi' \notin N(\mathbb D^2)$.
\end{proof}

We have now established that the group $H$, which is the kernel of Calabi on $\mathrm{Hameo}(\mathbb{D}^2,\omega)$, is not perfect.  To complete the proof of Theorem~\ref{thm:notsimple},  we will now 
produce a surjective group homomorphism from $H$ to $\mathbb{R}$.

\begin{proof} [Proof that there is a surjective homomorphism to $\mathbb{R}$.]

Choose the map $\Theta$ in Step $2$ above to commute with $T$ and to be generated by an autonomous Hamiltonian.  Then, $\psi'$, which recall is defined to be $T \circ \Theta^{-1}$, is the time-$1$ map of a one-parameter subgroup, say $\psi'_t$, because $T$ and  $\Theta$ are time-$1$ maps of commuting one-parameter subgroups. The image $N'$ of this $\psi'_t$ in the quotient $H'$ of 
$H$ by $N(\mathbb D^2)$ is isomorphic to $\mathbb{R}$.  We now claim that the identity map $N' \to N'$ extends to a homomorphism $H' \to N'$, which implies the existence of the desired surjective homomorphism, since $H$ surjects onto $H'$.  

To see that the identity map on $N'$ extends, note first of all that $N'$ is divisible.  Hence,  by a standard argument using Zorn's lemma (see e.g. \cite[Ch. XX, Sec. 4]{lang}), $N'$ is an injective object in the category of Abelian Groups, meaning that if $f: A \to N'$ is any homomorphism out of an abelian group $A$ and if $B$ is any abelian group containing $A$, then $f$ extends to $B$.  This implies that the identity map extends by taking $A = N'$, $f = \id $, and $B = H'$.  \end{proof}

\subsection{The case of the sphere}\label{s:sphere_case}
Recall, from Proposition \ref{prop:N-sphere},  the normal subgroup
$$N(\mathbb{S}^2) := \{ \psi \in \mathrm{Hameo}(\mathbb{S}^2, \omega) :  \text{the sequence }  \vert (2^k - 1) g_k(\psi) \vert \text{ is bounded}  \},$$
where $g_k$ is the quasimorphism defined by \eqref{eqn:gk}.  

We will show that $N(\mathbb{S}^2)$ is proper by showing that a variant $T'$ of the map $T$ from Proposition \ref{prop:T-hameo} is in $\mathrm{Hameo}(\mathbb{S}^2,\omega)$ but not in $N(\mathbb{S}^2)$.

\begin{proof}[Proof of properness of $N(\S^2)$]

Let $T'$ be the time-$1$ flow of $F(\theta, z) = f(z) = \sqrt{\frac{2}{1+z}}$, away from the south pole, and we set $T'(p_-) = p_-$.   We claim that $T' \in \mathrm{Hameo}(\mathbb{S}^2, \omega)$;  this follows directly from Proposition \ref{prop:T-hameo} via the observation that  $T' \circ T^{-1}$  is smooth. 

Define
\[ S(k) = \sqrt{2^k} \sum^{2^k - 1}_{i=1} \sqrt{ \frac{1}{i} }.\]
Then, arguing as in Step 2 of Section \ref{sec:disc},  it can be shown that
\begin{equation}\label{eqn:gk-T'}
( 2^{k} -1) g_k(T') = S(k) - \frac{2^k -1}{2^{k-1} - 1}S(k-1) = S(k) - 2 S(k-1) - \frac{1}{2^{k-1}-1}S(k-1).
\end{equation}  

We will only provide an outline of the proof of the above formula  as its derivation is similar to what was done  in Section \ref{sec:disc}. Here is the outline:   take an appropriate sequence of smoothings $F_n(z)$ of $F(z)$ which coincide with $F$ away from a small neighborhood of $p_-$.  Then, \eqref{eqn:gk-T'} follows from the following two items
\begin{enumerate}
\item $\phi^1_{F_n} \xrightarrow{C^0} T'$ and so, by the $C^0$ continuity of $g_k$, we have $g_k(T') = \lim_{n \to \infty} g_k(\phi^1_{F_n}).$
\item using the Lagrangian Control property, one obtains that 
$$g_k(\phi^1_{F_n}) = \frac{1}{2^k - 1} S(k) - \frac{1}{2^{k-1} - 1}S(k-1).$$
\end{enumerate}
The calculation in the second item, via the Lagrangian Control property, holds because we can assume that the link for $\mu_{2^k - 1}$ has the form $\lbrace z = -1 + \frac{i}{2^{k-1}}\rbrace$, where $i$ ranges from $1$ to $2^k - 1$.  For $z$ corresponding to such an $i$, $\sqrt{\frac{2}{1+z}} = \sqrt{\frac{2^k}{i}}$, and this is the value of $F_n$ on the link for sufficiently large $n$.

It follows from \eqref{eqn:gk-T'}  that to show that $T'$ is not in $N(\mathbb{S}^2)$, we need to estimate the difference
$$S(k) - 2 S(k-1) - \frac{1}{2^{k-1}-1}S(k-1).$$
The crux of the issue is showing that $S(k) - 2 S(k-1)$ is unbounded.  To see this, write
\begin{align*}
  S(k) - &2 S(k-1) = \sqrt{2^k} \sum^{2^k-1}_{i=1} \sqrt{\frac{1}{i}} - 2 \sqrt{2^k} \sum^{2^{k-1} - 1}_{i=1} \sqrt{\frac{1}{2i}}\\
& \ge \sqrt{2^k} \left( 1 - \sqrt{\frac{1}{2}}\right),
\end{align*}
which is unbounded in $k$.

To complete the proof, it therefore remains to show that the term $\frac{1}{2^{k-1}-1}S(k-1)$ is bounded in $k$.  To do this, we write 
\[ \frac{1}{2^{k-1}-1}S(k-1) = \frac{2^{k-1}}{2 \cdot 2^{k-1}-1}\frac{S(k-1)}{2^{k-2}}.\]
The term $\frac{S(k-1)}{2^{k-2}}$ differs from the right Riemann sum, for the integrable function $\sqrt{\frac{2}{1+z}}$ on $-1 \le z \le 1$, by $\frac{1}{2^{k-2}}$, hence $\frac{1}{2^{k-1}-1}S(k-1)$ is bounded in $k$.  We conclude from this the sequence $(2^k - 1) g_k(T')$ is unbounded and hence $T' \notin \mathrm{Hameo}(\mathbb S^2, \omega)$.  This completes the proof of Theorem \ref{thm:notsimple}.  
\end{proof}

The analogous argument as in the disc case shows from this that there is a surjection to $\mathbb{R}$.


\section{Infinitely many extensions of Calabi}\label{s:extension}

Having applied the two-term Weyl law to study the normal subgroup structure of $G = \Homeo_{\mathrm{c}}(\mathbb{D}^2,\omega)$, we now invoke related asymptotic considerations to prove Theorem~\ref{thm:cal}, which we recall for the reader states that the Calabi homomorphism admits infinitely many extensions to $G$.  We also elaborate on the promise from the introduction that this perspective has value in identifying new normal subgroups whose quotients can be computed.  We note for the benefit of the reader that while this section is thematically linked to the previous one, it does not cite results from there and so can be read independently.

\subsection{The main theorem}

We begin with the promised proof of Theorem~\ref{thm:cal}, which collects considerations of the asymptotics of the $f_k$ via a short argument. 

\begin{proof}[Proof of Theorem~\ref{thm:cal}]
Define the group $R' := \mathbb{R}^{\mathbb{N}} / \sim$, where $s \sim t$ if and only if $s - t$ has limit $0$.  
There is a natural map 
\begin{equation}
\label{eqn:defns} 
S: G \to R', \quad \quad g \to (f_2(g), f_3(g), \ldots, f_n(g), \ldots).
\end{equation}
(We have not included $f_1$ here, because as we have defined it, it is $0$.)  
By Lemma~\ref{lem:defect}, this is a group homomorphism.
There is also a canonical homomorphism 
\[ \Delta: \mathbb{R} \to R', \quad x \mapsto (x,x,\ldots,x).\]

Now by the Weyl law \eqref{eqn:weyl}, we have
\begin{equation}
\label{eqn:exten}
 S(h) = (\Cal,\Cal,\Cal,\ldots)
\end{equation}
for every $h\in\Diff_{\mathrm{c}}(\D^2,\omega)$.

We now find a section of the map $\Delta$, as follows.  The group $R'$ is a vector space over $\mathbb{R}$.  Take the vector $v_1 = (1,\ldots,1,\ldots) \in R'$;  by Zorn's Lemma, we can extend this to a basis $\beta$ for $R'$.  The section of $\Delta$ now comes from the splitting of $R'$ with respect to this basis.    More precisely, we define 
\[ s: R' \to \mathbb{R}, \quad s(v) = a_1, \quad \quad v = a_1 v_1 + \sum_{v_i \in \beta, v_i \ne v_1} a_i v_i.\]
It now follows from \eqref{eqn:exten} that $s \circ S$ is the desired extension.  Since there are infinitely many choices of extensions $\beta$, and the map $S$ is surjective (see Proposition~\ref{prop:surj} below), it follows that there are infinitely many extensions.  
\end{proof}

\begin{rem}
\label{rem:zorn}
One might wonder to what degree 
Zorn's lemma is actually necessary in extending 
the Calabi homomorphism to a group homomorphism $\Homeo_c(\mathbb{D}^2,\omega)\to \mathbb{R}$.  As communicated to the authors by C. Rosendal, there are models of set theory where the axiom of choice is false and every homomorphism between Polish groups is continuous; in particular, no extension of $\Cal$ to a group homomorphism $\Homeo_c(\mathbb{D}^2,\omega)\to \mathbb{R}$ 
exists
in those models. 
See \cite[Thm. 5.]{Rosen} and \cite{Rosen2}.  
\end{rem}

\begin{rem}
\label{rem:ext}
In \cite{CGHMSS21}, the Calabi invariant was previously extended to a homomorphism $Cal$ on $\Hameo(\D^2,\omega)$ by the rule \[ H \mapsto \int H \omega dt,\]
where $H$ is any Hamiltonian for a given hameomorphism; such a Hamiltonian is not unique, but \cite{CGHMSS21} showed that this extension does not depend on the choice of Hamiltonian.  Any of the extensions to $\Homeo_{\mathrm{c}}(\D^2,\omega)$ in Theorem~\ref{thm:cal} agree with this extension when restricted to Hameo; this follows from the fact that, similarly to \eqref{eqn:exten},
\begin{equation}
\label{eqn:calham}
 S(h) = (\Cal,\Cal,\Cal,\ldots),
\end{equation}
for any $h\in\Hameo(\D^2,\omega)$.

To see why the above equation is true, we note that, as in the proof of \cite[Thm.\ 1.1]{CGHMSS21}, if $h \in \Hameo(\D^2,\omega)$ and $H:[0,1]\times \D^2\to\R$ is a $C^0$ Hamiltonian for $h$, then, for any $\varepsilon > 0$, we can find smooth Hamiltonians $G_m$ such that 

\begin{enumerate}
\item $\phi_{G_m}^1$ converges to $h$ in the $C^0$ topology,
 \item $G_m$ uniformly converges to  $H$.
\end{enumerate}

Then, 

\[ |f_k(h) - f_k(G_m)| < \varepsilon \quad\text{and}\quad \left| \int G_m - \int H \right| < \varepsilon,\]
where in the first inequality above we have used the Hofer continuity property; see \ref{thm:mu-k}.  Since $\varepsilon > 0$ is arbitrary, \eqref{eqn:calham} follows from the above inequalities and \eqref{eqn:weyl}.  

\end{rem}

\subsection{Normal subgroups with explicit quotients}
\label{sec:newnormal}

It has been an open question since the proof of the simplicity Conjecture mentioned in the introduction to identify the quotient of $G$ by the normal subgroup of finite energy homeomorphisms $\mathrm{FHomeo}$ constructed there; see \cite{CGHS20}.  The circle of ideas around the proof of Theorem~\ref{thm:cal} allows us to resolve a variant of this question: we can find normal subgroups whose quotient can be calculated.

For example, define $N$ to be the kernel of the map $S$ from \eqref{eqn:defns}.  

\begin{prop}
\label{prop:surj}
The map $S$ from \eqref{eqn:defns} is surjective.  In particular, $G/N \simeq R'$.
\end{prop}

\begin{proof}

Given an element $s \in \mathbb{R}^{\mathbb{N}}$, we define a smooth autonomous Hamiltonian $H$ on the complement of the north pole $p_+ \in \S^2$, and depending only on $z$, recursively as follows.  

Call $s_i$ the $(i-1)^{st}$ component of $s$.  To motivate what follows, note that, given $k$, we can take our Lagrangian Link to correspond to the set $\lbrace z = -1 + \frac{2i}{k+1} : 1 \le i \le k \rbrace.$
Fix also the data of a smooth function $E:[0,1] \to \mathbb{R}$ that is constant near $0$ and $1$ and satisfies $E(0) = 0$ and $E(1) = 1$.

We start by defining $H$ to be equal to $0$ on $\lbrace -1 \le z \le 0 \rbrace$.
Next, we define $H$ to be equal to $2 s_2$ on $\lbrace z = 1/3 \rbrace$.  We now extend $H$ to a smooth function on $\lbrace -1 \le z \le 1/3 \rbrace$ by defining it to be equal to $ H(1/3) E( \frac{z}{1/3} )$ on $\lbrace 0 \le z \le 1/3 \rbrace.$  Note that for any extension $H'$ of $H$ to a smooth function of $z$ on the entire interval $[-1,1]$,  we have $f_2(H')=s_2$,
by the Lagrangian Control property (\ref{eq:lag-control-fk}). 

Now assume, inductively, that we have extended 
$H$ to a smooth function on $\lbrace -1 \le z \le -1 + \frac{2k}{k+1} \rbrace$, for some $k \ge 2$, that is constant near the endpoints of this interval and satisfies 
$$f_i(H) = s_i, \; \; 2 \leq i \leq k$$
for any further extension of $H$ to a smooth function on $[-1,1]$.  We seek to extend $H$ to a smooth function on $\lbrace -1 \le z \le -1 + \frac{2k+2}{k+2} \rbrace$ that is also constant near the endpoints of this interval and satisfies 

$$f_i(H) = s_i, \; \; 2 \leq i \leq k+1$$
for any further extension of $H$ to a smooth function on $[-1,1]$.  Note, first of all, that $-1 + \frac{2k}{k+2} \le -1 + \frac{2k}{k+1}$.  In particular, the equation
\[ H(-1 + \tfrac{2k+2}{k+2}) =  (k+1) s_{k+1} - \sum^{k}_{i=1} H(-1 + \tfrac{2i}{k+2}),\]
makes sense, and we use it to define $H$ on $\lbrace z =  -1 + \frac{2k+2}{k+2} \rbrace.$  We therefore have a function $H$ defined on $\lbrace -1 \le z \le -1 + \frac{2k}{k+1} \rbrace \cup \lbrace z = -1 + \frac{2k+2}{k+2} \rbrace$, which is smooth on the first of these sets and constant near the endpoints of the first of these sets.  Since $-1 + \frac{2k}{k+1} < -1 + \frac{2k+2}{k+2},$ there is no obstruction to further extending $H$ smoothly to $\lbrace -1 \le z \le -1 + \frac{2k+2}{k+2} \rbrace:$ more precisely, we define $H$ to be 
\[ \left(H(-1 + \frac{2k+2}{k+2}) - H(-1 + \frac{2k}{k+1}) \right) E\left( \frac{ z - (-1 + \frac{2k}{k+1}) }{\frac{2k+2}{k+2} - \frac{2k}{k+1} } \right) + H(-1 + \frac{2k}{k+1})\]
 on $\lbrace -1 + \frac{2k}{k+1} \le z \le -1 + \frac{2k+2}{k+2} \rbrace.$
  
As above, we note that any further extension of $H$ to a smooth function on $\lbrace -1 \le z \le 1 \rbrace$ will have
\begin{equation}
\label{eqn:key-induction}
f_{k+1}(H)=s_{k+1},
\end{equation} 
by the Lagrangian Control property (\ref{eq:lag-control-fk}).

Given an element $s \in \mathbb{R}^{\mathbb{N}}$, we now define $\psi$ to be the time-$1$ flow of the Hamiltonian $H$ constructed above, away from $p_+$, and we set $\psi(p_+) = p_+$.  We can view this as a compactly supported homeomorphism of the disc, which we also denote by $\psi$, and we claim that $S(\psi) = s:$ indeed, for any fixed $k$, we can approximate $\psi$ in $C^0$ by smooth flows corresponding to Hamiltonians that depend only on $z$, without changing the values of $H$ on the components $\lbrace z = -1 + \frac{2i}{k+1} \rbrace$ of the Lagrangian Link, hence the claim follows from \eqref{eqn:key-induction} together with the $C^0$ continuity of  $f_k$ (Theorem ~\ref{th:C0continuity-fk}).  
\end{proof}

\begin{rem}
For a more familiar presentation of $G/N$ via Proposition~\ref{prop:surj}, we note that the group $R'$ is isomorphic to $\mathbb{R}$.  Indeed, both are uncountable vector spaces over $\mathbb{Q}$ of the same cardinality.
\end{rem}

\begin{rem}
The map $S$ allows us to define many other subgroups whose quotients can be identified.  Indeed, we can take any subgroup $H \subset R'$, and then by Proposition~\ref{prop:surj}, $N_H := S^{-1}(H)$ will be a normal subgroup with quotient $H$.  One can think of the different $N_H$ as ``leading asymptotics subgroups": they correspond to different prescriptions of the leading asymptotics of the $f_k$.  We may also produce groups by varying the target of $S$ by taking different quotients of $\mathbb{R}^\mathbb{N}$.  For example, if we quotient by the relation that $s \sim t$ if and only if $s-t$ remains bounded and take this to be the target of $S$, then the induced homomorphism out of $G$ is still surjective, but one can show that its kernel contains $\FHomeo$ and $\Hameo,$ as introduced in the discussion at the end of the introduction.
\end{rem}


\section{The commutator group of $G$ is simple}\label{s:simple}

The goal of this section is to prove Theorem \ref{thm:comm}. We denote by $G$ the kernel of the mass flow homomorphism $\Homeo_{\mathrm{c}}(\Sigma, \omega)\to \R$, where $\Sigma$ is a surface either compact or the interior of a compact surface with boundary. We denote by $[G,G]$ the commutator subgroup, i.e.\ the subgroup generated by commutators. We will denote by $[f,g]=f^{-1}g^{-1}fg$ the commutator of two elements $f$ and $g$. Theorem \ref{thm:comm} asserts that $[G,G]$ is simple.

As was mentioned in the introduction, it is known (see footnote \ref{footnote:[G,G]}) that any normal subgroup of $G$ contains $[G,G]$ 
and in particular 
 the commutator group of $[G,G]$, which is normal in $G$, contains $[G,G]$, hence 
 $[G,G]$ is perfect.
Another consequence of this fact is  the simplicity of $[G,G]$ (Theorem \ref{thm:comm}) follows from the next lemma. 

\begin{lem}\label{lem:normalGG-normalG} Any normal subgroup of $[G,G]$ is normal in $G$.
\end{lem}

\begin{proof} Let $H$ be a normal subgroup of $[G,G]$. To prove that $H$ is normal in $G$, we need to prove that for all $h\in H$ and $g\in G$, the conjugate $g^{-1}hg$ belongs to $H$. 
\paragraph{\textbf{First step.}}  Let $h\in H$ and $g\in G$. We will prove that  $g^{-1}hg\in H$ under the condition:
\[
\tag{I} \text{The open set $U=\Sigma\setminus \supp(h)$ is non empty.}
\]

Let $S$ be an embedded disc included in $U$. We will show that $g$ may be written as a composition $g=a b$, with $a$ supported in $S$ and  $b\in[G,G]$. Since $a$ and $h$ have disjoint supports, we have $a^{-1}ha=h$. Using that $H$ is normal in $[G,G]$ we deduce $g^{-1}hg=b^{-1}hb\in H$ as claimed.

To prove the decomposition of $g$,  
let $(S_i)_{i\in I}$ be a finite open cover of $\Sigma$ by discs of the form $S_i=f_i(S)$ for some $f_i\in G$. By Fathi's fragmentation theorem \cite[Thm. 6.6]{fathi} the map $g$ can be written as a product $g=g_1\cdots g_N$ of elements in $G$ such that each $g_\alpha$ is supported in a disc $S_{i_\alpha}$ for some $i_\alpha\in I$. 
Such $g_\alpha$ can then be written in the form
\[g_{\alpha}=(f_{i_\alpha}^{-1}g_\alpha f_{i_\alpha})\circ [f_{i_\alpha},g_\alpha].\]
Since $f_{i_\alpha}^{-1}g_\alpha f_{i_\alpha}$ is supported in $S$, this shows that each $g_\alpha$ may be represented in the quotient $G/[G,G]$ by an element supported in $S$.
As a consequence, their product $g$ may also be represented in $G/[G,G]$ by an element $a$ supported in $S$. This exactly means that $g=ab$ for some $b\in[G,G]$.

\medskip
\paragraph{\textbf{Second step.}} We finally show that $g^{-1}hg\in H$ for any $h\in H$ and $g\in G$. This will rely on the first step and the following lemma. 

\begin{lem}\label{lem:lem-in-simplicity} Let $h\in H$ and let $z$ be a fixed point of $h$ (which exists by Arnold conjecture \cite{Matsumoto}). Then for every sufficiently small open neighborhood $U$ of $z$, there exists  $\ell \in H$ such that $\supp(\ell)\neq\Sigma$ and $\ell$ coincides with $h$ on $U$.  
\end{lem}

We postpone the proof of this lemma and use it to conclude the proof of the second step. Let $h\in H$ and $g\in G$. Let $\ell\in H$ be as provided by Lemma \ref{lem:lem-in-simplicity}. Then, $h\ell^{-1}$ and $\ell$ belong to $H$ and both satisfy condition (I). Our first step shows that $g^{-1}(h\ell^{-1})g\in H$ and $g^{-1}\ell g\in H$. 
As a consequence, their product $g^{-1}hg$ belongs to $H$. This concludes the proof of the second step and of Lemma \ref{lem:normalGG-normalG}.
\end{proof}

\begin{proof}[Proof of Lemma \ref{lem:lem-in-simplicity}.] 
  Let $U$ be a small neighborhood of $z$. How small it is will be made precise below. Since $z$ is fixed, it is known\footnote{This is a standard folklore statement that can be proved by a combination of the Schoenflies and the Oxtoby-Ulam theorem.} that for every open neighborhood  $V$ of $z$, there exists an element $\alpha\in G$ which coincides with $h$ in a neighborhood of $z$ and is supported in $V$. We may assume that $U$ is so small that $\alpha=h$ on $U$. We will use such an $\alpha$ to build our map $\ell$.

  Let $x$ be a point such that $h(x)\neq x$. Note that we may assume without loss of generality that such a point exists. Taking a point $y$ close to $x$ but distinct from $x$, we obtain a configuration of four pairwise distinct points $x$, $y$, $h(x)$, $h(y)$. Let $f\in G$ be such that $f(x)=y$. Let $A$ be an open neighborhood of $x$. If $A$ is chosen small enough, then the four open sets $A$, $B=h(A)$, $C=f(A)$ and $D=h(C)$ are pairwise disjoint. In this situation, it is easy to check that for any $g\in G$ supported in $A$ we have
  \[\supp([f^{-1},g])\subset A\cup C\quad\text{ and }\quad [f^{-1},g]=g \text{ on }A.\] 
Similarly, since $B\cup D=h(A\cup C)$, we have for any $g\in G$ supported in $A$
\[\supp([h^{-1},[f^{-1},g]])\subset A\cup B\cup C\cup D\quad\text{ and }\quad [h^{-1},[f^{-1},g]]=g \text{ on }A.\]
Since $h\in H$ and $H$ is normal in $[G,G]$, the element $[h^{-1},[f^{-1},g]]$ belongs to $H$. Thus, we have shown that any element of $G$ supported in $A$ coincides on $A$ with an element of $H$ supported in $A\cup B\cup C\cup D\neq \Sigma$. We will apply this fact to an appropriate conjugate of the map $\alpha$ from the beginning of the proof.

Let $\beta\in[G,G]$ be a map that sends $z$ to $x$ (this can be found for instance  among diffeomorphisms). Then, if the open sets $U$ and $V$ are chosen sufficiently small, the map $\beta\alpha\beta^{-1}$ is supported in $A$. By the above observation, there exists an element $\gamma\in H$ which coincides with $\beta\alpha\beta^{-1}$ on $A$ and whose support is not the whole of $\Sigma$.
Then, $\ell=\beta\gamma\beta^{-1}$ suits our needs. Indeed, $\ell$ coincides with $\alpha$ on $V$ hence with $h$ on $U$. Moreover, since $H$ is normal in $[G,G]$, $\ell\in H$ and its support is not the whole of $\Sigma$. 
\end{proof}

\def\cprime{$'$}

{\small

\medskip
\noindent Dan Cristofaro-Gardiner\\
\noindent Mathematics Department, University of Maryland, College Park, USA.\\
{\it e-mail}: dcristof@umd.edu
\medskip

\medskip
\noindent Vincent Humili\`ere \\
\noindent Sorbonne Université and Université de Paris, CNRS, IMJ-PRG, F-75006 Paris, France\\
\noindent \& Institut Universitaire de France.\\
{\it e-mail:} vincent.humiliere@imj-prg.fr
\medskip

\medskip
\noindent Cheuk Yu Mak\\
\noindent School of Mathematical and Physical Sciences, Hicks Building, University of Sheffield, Sheffield, S10 2TN, UK\\
{\it e-mail:} c.mak@sheffield.ac.uk

\medskip
 \noindent Sobhan Seyfaddini\\
\noindent Department of Mathematics, ETH Zürich, Rämistrasse 101, 8092, Zürich,
Switzerland. \\
 {\it e-mail:}  sobhan.seyfaddini@math.ethz.ch.

 \medskip
 \noindent Ivan Smith\\
\noindent Centre for Mathematical Sciences, University of Cambridge, Wilberforce Road, CB3 0WB, U.K.\\
{\it e-mail:} is200@cam.ac.uk

}

\end{document}